\documentclass[global]{svjour}
\usepackage{amsfonts}
\usepackage{amsmath}

\input{tcilatex}
\begin{document}

\journalname{Journal of Theoretical Probability}
\author{Vladimir Dobric\inst{1} \and Lisa Marano\inst{2}}
\institute{Lehigh University 
\email{vd00@lehigh.edu}
\and West Chester University of Pennsylvania 
\email{lmarano@wcupa.edu}%
}
\title{{Lower Bounds for the Distribution of Suprema of Brownian Increments
and Brownian Motion Normalized by the Corresponding Modulus Functions}}
\maketitle

\begin{abstract}
The L\'{e}vy-Ciesielski Construction of Brownian motion is used to determine
non-asymptotic estimates for the maximal deviation of increments of a
Brownian motion process $(W_{t})_{t\in \left[ 0,T\right] }$ normalized by
the global modulus function, for all positive $\varepsilon $ and $\delta $.
Additionally, uniform results over $\delta $ are obtained. Using the same
method, non-asymptotic estimates for the distribution function for the
standard Brownian motion normalized by its local modulus of continuity are
obtained. Similar results for the truncated Brownian motion are provided and
play a crucial role in establishing the results for the standard Brownian
motion case.
\end{abstract}


\section{Introduction}

We present a unified method for establishing both local and global moduli of
continuity for a Brownian motion process, $(W_{t})_{t\geq 0}$.  A most
useful conequence of our process allows for explicit estimates which will be
explained in more detail below. We briefly recall some basic properties of
the L\'{e}vy-Ciesielski construction in section two. This construction
compared to others is used most often to establish the continuity of sample
paths; however, not until recently has it been exploited to show other
properties. For example, J.P. Kahane in 1985 applied the orthonormal
expansion to study slow and fast points of the Brownian motion process \cite%
{Kahane}. M. Pinsky offers a simple proof of the existence of the modulus of
continuity based on the L\'{e}vy-Ciesielski construction \cite{pinsky}. 

In section three, we exploit the L\'{e}vy-Ciesielski construction and the
piecewise-linear truncated process $\left( W_{t}^{n}\right) _{t\in \lbrack
0,1]}$ (the following section contains an explicit description of this
process) over dyadic intervals to establish several results regarding the
global modulus of continuity for Brownian motion.

Specifically for every $\varepsilon >0$\ and for every $\delta >0$, we
determine an estimate for the maximal deviation of increments of a Brownian
motion process, $W_{t}$, normalized by the global modulus function. More
explicitly, we have determined functions $k$ and $p_{1}$ so that for every $%
\varepsilon >0$ and every $\delta >0$ 
\begin{equation*}
\mathbb{P}\left( \underset{\left\vert t-s\right\vert <\delta }{\sup_{0\leq
t<s\leq 1}}\frac{\left\vert W_{t}-W_{s}\right\vert }{g(\delta )}\leq
1+k(\varepsilon ,\delta )\right) \geq 1-p_{1}(\varepsilon ,\delta ),
\end{equation*}%
where $g(x)$ is the global modulus of continuity for Brownian motion, $\sqrt{%
2xL\frac{1}{x}}$. More significantly, we also establish uniform results over 
$\delta $ for the global modulus of continuity. Specifically, we determine $%
p_{2},$ so that for all $\varepsilon >0$ and every $\delta _{o}>0,$%
\begin{equation*}
\mathbb{P}\left( \sup_{\delta \leq \delta _{o}}\underset{\left\vert
t-s\right\vert <\delta }{\sup_{0\leq t<s\leq 1}}\frac{\left\vert
W_{t}-W_{s}\right\vert }{g(\delta )}\leq 1+k(\varepsilon ,\delta
_{o})\right) \geq 1-p_{2}(\varepsilon ,\delta _{o}).
\end{equation*}%
Surprisingly, $p_{2}$ is similar to $p_{1}$ but with larger constants when $%
\varepsilon $ is bounded away from 0. It should not be a surprise that to
establish the uniform result was computationally challenging to say the
least.

Using a similar method, in section four, we establish the local modulus of
continuity for the standard Brownian motion, with the same type of gains.
That is, we construct functions $l$ and $q$ so that 
\begin{equation*}
\mathbb{P}\left( \sup_{t\leq \delta }\frac{W_{t}}{h\left( t\right) }\leq
1+l(\varepsilon ,\delta )\right) \geq 1-q(\varepsilon ,\delta )
\end{equation*}%
where $h(x)=\sqrt{2x\ln \ln \frac{1}{x}}$, the local modulus of continuity.
In all cases the L\'{e}vy-Ciesielski representation reveals how and why the
logarithmic terms appear in the corresponding $g$ and $h$ functions.

We stress the differences between our work and others. First, the results of
this paper are not asymptotic results, for they hold for every $\varepsilon
>0$\ AND for every $\delta >0$. Many papers have been written on this
subject; to the best of our knowledge, all of which are of an
asymptotic-type. See \cite{kolmogorov}, \cite{Erdos}, and \cite{levy}, and
more recently, \cite{Khosh}, \cite{Einmahel}, and \cite{pinsky}. Their
results take the form: for every $\varepsilon >0$, there is $\delta \left(
\varepsilon \right) >0 \dots$ Transitioning from these asymptotics to the
results presented here would require several steps of approximations, as
compared to our one.{\Large \ }First $\delta \left( \varepsilon \right) $
must be estimated. Then, after rescaling, a second level of approximations
would be required to express the LIL\ or modulus of continuity in terms of
an arbitrary $\delta .$ These multiple levels of estimation would certainly
affect the constants involved in the asymptotic results. Our results are
straightforward. We express Brownian motion by an appropriate infinite sum,
then determine 
\begin{equation*}
\underset{\left\vert t-s\right\vert <\delta }{\sup_{0\leq t<s\leq 1}}\frac{%
\left\vert W_{s}^{n}-W_{t}^{n}\right\vert }{g\left( s-t\right) }
\end{equation*}%
and 
\begin{equation*}
\sup_{t\leq \delta }\frac{W_{t}^{n}}{h\left( t\right) }
\end{equation*}%
exactly; our only estimate is of the tail. Second, our approach is unique in
that the method works for both the global and local modulus and allow us to
easily establish both L\'{e}vy's modulus of continuity and the law of
iterated logarithms. The only difference between the two proofs is where we
split the process into a piecewise-linear truncated process and an infinite
tail. Finally, in practice, the usefulness of our results is the ability to
choose $\delta $ a priori and independently of $\varepsilon $, and to select 
$\varepsilon $ afterwards corresponding to a desired confidence level.

\begin{remark}
Maple\texttrademark\ was used in many of our calculations. The values
obtained were rounded to at most three decimal places in a way which did not
compromise the direction of any inequalities and was coarser than the
precision level of the computer algebra system. Moreover, as we seek to
determine estimates for probabilities we at times produce long strings of
inequalities. Thus when we number an inequality we are referring to greatest
quantity in the string.
\end{remark}

\begin{remark}
Throughout the paper, we will use the convention: $L\left( x\right) :=\ln x$
and $L_{2}\left( x\right) =\ln \ln x;$ we also use the standard notation, $%
\left[ x\right] $ to denote the greatest integer less than or equal to $x.$
\end{remark}

\section{The L\'{e}vy-Ciesielski Construction}

Throughout this paper, we let $W=$ $(W_{t})_{0\leq t\leq 1}$, a Brownian
motion process over the unit interval and $t\rightarrow W_{t}$ be a
realization of the process over the unit interval. The L\'{e}vy-Ciesielski
construction of Brownian motion is based on the Haar expansion of the
covariance function of a Brownian motion process $W$ in the Cameron-Martin
space. Via an isomorphism, it leads to the following representation of the
Brownian motion process:

\begin{equation*}
W_{t}=tX_{o}+\sum_{j=0}^{\infty }2^{\frac{-j}{2}}\sum_{k=0}^{2^{j}-1}\Lambda
_{j,k}\left( t\right) X_{j,k}
\end{equation*}%
where $X_{o}$ and $X_{j,k}$, for all$\ j,k$, are independent, standard
normal random variables and 
\begin{equation*}
\Lambda _{j,k}(t)=\min \{2^{j}t-k,1-2^{j}t+k\}\mathbf{1}_{I_{j,k}}\left(
t\right)
\end{equation*}%
with $I_{j,k}=\left[ k2^{-j},\left( k+1\right) 2^{-j}\right) .$

Let $W_{t}^{n}$ be the $n^{th}$ partial sum of $W_{t},$ which includes $%
tX_{o}$. The process $W^{n}=(W_{t}^{n})$ possesses some interesting
properties. First, $W_{t}^{n}$ and $W_{t}$ agree at the dyadics at the $%
\left( n+1\right) ^{th}$ level; that is, $W_{k2^{-n-1}}^{n}=W_{k2^{-n-1}}$
for $k=0, \dots ,2^{-n-1}.$ Moreover, for $t\in I_{n+1,k},$ the process $%
W_{t}^{n}$ is linear in $t,$ i.e. $W_{t}^{n}=At+B$ where $A$ and $B$ are
normal random variables. Therefore, the process, $W^{n}$ is equivalent to
the piecewise-linear process, $\overline{W}^{n},$ created by connecting the
points $\left( k/2^{n+1},W_{k/2^{n+1}}\right) ,$ $k=0, \dots ,2^{n+1}$
linearly. As mentioned above this was noticed by P. L\'{e}vy and we shall
use this fact repeatedly throughout. A more thorough introduction to this
expansion can be found in \cite{Steele}.

\section{Global Maximal Deviations for Truncated Brownian Increments and
Brownian Increments}

In this section we develop several results regarding the global modulus of
continuity for the truncated Brownian motion process and the process itself.
First we obtain an estimate for the distribution function of the ratio
between the truncated Brownian increment and the global modulus of
continuity function $g\left( \delta \right) $. Using this result we
establish an estimate for the distribution function of the maximal deviation
for the ratio of the Brownian increment and $g\left( \delta \right) .$ More
specifically, for $\varepsilon ,\delta >0,$ we determine the probability of
the set 
\begin{equation*}
\left\{ \underset{\left\vert t-s\right\vert \leq \delta }{\sup_{0\leq
t<s\leq 1}}\frac{\left\vert W_{s}-W_{t}\right\vert }{g(\delta )r\left(
\delta \right) }\leq \sqrt{1+\varepsilon }\right\}
\end{equation*}%
where $r\left( x\right) =1+3.5\left( L\frac{1}{\left\vert x\right\vert }%
\right) ^{-\frac{1}{2}}.$ As $\delta \rightarrow 0,$ the function $\delta
\rightarrow \underset{\left\vert t-s\right\vert \leq \delta }{\sup }\frac{%
\left\vert W_{s}-W_{t}\right\vert }{g\left( \delta \right) r\left( \delta
\right) }$ is not necessarily monotonic. Therefore we establish an estimate
for the probability of the set 
\begin{equation*}
\left\{ \sup_{\delta \leq \delta _{o}}\underset{\left\vert t-s\right\vert
\leq \delta }{\sup_{0\leq t<s\leq 1}}\frac{\left\vert W_{s}-W_{t}\right\vert 
}{g(\delta )r\left( \delta \right) }\leq \sqrt{1+\varepsilon }\right\} ,
\end{equation*}%
which is monotonic in $\delta _{o}$.

\subsection{Preliminaries}

The first lemma is essential in estimating the probability of the set 
\begin{equation*}
\left\{ \underset{\left\vert t-s\right\vert \leq \delta }{\sup_{0\leq
t<s\leq 1}}\frac{\left\vert W_{s}^{n}-W_{t}^{n}\right\vert }{g\left(
s-t\right) }>\sqrt{1+\varepsilon }\right\} .
\end{equation*}%
The second is needed to uniformly estimate the tail of the truncated
increment.

Notice that $W_{s}^{n}$ and $W_{t}^{n}$ are piecewise linear in $s$ and $t$,
and therefore so is their difference; that is, $W_{s}^{n}-W_{t}^{n}=At+Bs+C$
for some random variables $A,B,$ and $C.$

\begin{lemma}
\label{mod calc lemma}Let $\delta _{0}>\delta \geq 0$ and $f:C\rightarrow R$
be defined by 
\begin{equation*}
f(t,s)=\frac{\left\vert at+bs+c\right\vert }{g\left( s-t\right) }
\end{equation*}%
where $a,b$ and $c$ are constants and $g$ is the global modulus of
continuity function of a Brownian motion process and $C$ is the convex set 
\begin{equation*}
\left\{ \lbrack t_{1},t_{2}]\times \left[ s_{1},s_{2}\right] \cap \left\{
\left( t,s\right) :\delta \leq \left\vert s-t\right\vert \leq \delta
_{0}\right\} \right\} 
\end{equation*}%
with $t_{1},s_{1}>0,t_{2}<s_{1},$and $s_{2}-t_{1}<1.$Then $f$ achieves its
maximum at one of the extreme points of the convex set $C.$

Moreover, the supremum of $f$ over all $s$ and $t$ such that $\left\vert
s-t\right\vert <\delta $ is achieved at a value for $s-t$ which is bounded
away from zero. That is, for $\delta >0,$%
\begin{equation*}
\underset{\left\vert s-t\right\vert \leq 2^{-n}\wedge \delta }{%
\sup_{(t,s)\in I_{n+1,k}\times I_{n+1,k+1}}}\frac{\left\vert
at+bs+c\right\vert }{g\left( s-t\right) }=\underset{2^{-n-1}\leq \left\vert
s-t\right\vert \leq 2^{-n}\wedge \delta }{\sup_{(t,s)\in I_{n+1,k}\times
I_{n+1,k+1}}}\frac{\left\vert at+bs+c\right\vert }{g\left( s-t\right) }.
\end{equation*}
\end{lemma}

\begin{proof}
Calculus.
\end{proof}

The next lemma is a variation of a well-known fact from the theory of
Gaussian processes: the maximum of a finite collection of identically
distributed normal random variables essentially grows as the square root of
the natural log of the cardinality of the collection \cite{talagrand}. Many
variations of this lemma have been used in the past. For instance, in 1991,
a version similar to the one presented here was used by Meyer \cite{meyer}.

\begin{lemma}
\label{tail estimate}For $d>0$%
\begin{equation*}
\mathbb{P}\left( \underset{0\leq k\leq 2^{j}-1}{\max_{j\geq n}}\frac{%
\left\vert X_{j,k}\right\vert }{\sqrt{L2^{j}}}>\sqrt{2(d+1)}\right) \leq 
\frac{2^{-dn}}{\left( 1-2^{-d}\right) \sqrt{\pi L2^{n}}}.
\end{equation*}
\end{lemma}

\begin{proof}
\begin{align*}
\mathbb{P}\left( \underset{0\leq k\leq 2^{j}-1}{\max_{j\geq n}}\frac{%
\left\vert X_{j,k}\right\vert }{\sqrt{L2^{j}}}>\sqrt{2(d+1)}\right) & \leq
\sum_{j=n}^{\infty }2^{j}\mathbb{P}\left( \frac{\left\vert N\left(
0,1\right) \right\vert }{\sqrt{L2^{j}}}>\sqrt{2(d+1)}\right) \\
& \leq \sum_{j=n}^{\infty }\frac{2^{-dj}}{\sqrt{\pi (d+1)L2^{j}}} \\
& \leq \frac{2^{-dn}}{\left( 1-2^{-d}\right) \sqrt{\pi L2^{n}}}.
\end{align*}
\end{proof}

\subsection{Global maximal deviations}

In this subsection we establish three results. The first theorem estimates
the distribution function of the maximal deviation between the ratio of the
truncated increment and the global modulus function for a fixed $\delta >0$.
Based on that result, coupled together with the tail estimate, the second
theorem estimates the distribution function for the maximal deviation of the
ratio of the increment of the process $W$ and the global modulus of
continuity function for a fixed $\delta .$ Finally, we go beyond the results
of the standard modulus of continuity by establishing a result which holds
uniformly over $\delta $. This last result allows us to establish rates of
convergence; see \cite{DobricMarano}.

\begin{theorem}
\label{mod theorem nth partial sum}For $\varepsilon >0$ and $n\geq 4$, if $%
0<\delta <2^{-n-1}$ we have 
\begin{equation*}
\mathbb{P}\left( \underset{\left\vert s-t\right\vert \leq \delta }{%
\sup_{0\leq t<s\leq 1}}\frac{\left\vert W_{s}^{n}-W_{t}^{n}\right\vert }{%
g\left( s-t\right) }>\sqrt{1+\varepsilon }\right) \leq \frac{3\delta
^{\varepsilon }}{\sqrt{\pi L\frac{1}{\delta }}},
\end{equation*}%
and if $\delta \geq 2^{-n-1}$ then 
\begin{equation*}
\mathbb{P}\left( \underset{\left\vert s-t\right\vert \leq \delta }{%
\sup_{0\leq t<s\leq 1}}\frac{\left\vert W_{s}^{n}-W_{t}^{n}\right\vert }{%
g\left( s-t\right) }>\sqrt{1+\varepsilon }\right) \leq \frac{2^{-\varepsilon
\left( n+1\right) }K\left( \varepsilon ,\delta ,n\right) }{\sqrt{\pi L\frac{1%
}{\delta }}},
\end{equation*}%
where 
\begin{equation*}
K\left( \varepsilon ,\delta ,n\right) =1+9\left( 2\right) ^{\varepsilon
}+4\left( 2^{n+1}\delta \right) ^{1+\varepsilon }+2\left( 2^{n+1}\delta
\right) ^{2+\varepsilon }.
\end{equation*}
\end{theorem}

\begin{remark}
\label{restriction on n}The restriction on $n$ is imposed to assure
monotonicity of the functions involved. Without this restriction the
constants involved would be greater.
\end{remark}

\begin{proof}
Let $\varepsilon ,\delta >0$ and $n\geq 4.$ Define $\delta _{n}=\min \left\{
\delta ,2^{-n-1}\right\} $. Let $I_{k}$ denote the $k^{th}$ dyadic interval
at the level $n+1;$ that is, $I_{k}=I_{n+1,k}.$ While $W_{s}^{n}-W_{t}^{n}$
is piecewise linear over the set $\left\{ \left( t,s\right) |0\leq t<s\leq
1,\left\vert s-t\right\vert \leq \delta \right\} $, the increment is linear
in $s$ and $t$ when $(t,s)\in I_{k}\times I_{k+l}$. Therefore, we increase
the size of the set 
\begin{equation*}
\left\{ \underset{\left\vert s-t\right\vert \leq \delta }{\sup_{0\leq
t<s\leq 1}}\frac{\left\vert W_{s}^{n}-W_{t}^{n}\right\vert }{g\left(
s-t\right) }>\sqrt{1+\varepsilon }\right\}
\end{equation*}%
to gain linearity and obtain 
\begin{align}
& \mathbb{P}\left( \underset{\left\vert s-t\right\vert \leq \delta }{%
\sup_{0\leq t<s\leq 1}}\frac{\left\vert W_{s}^{n}-W_{t}^{n}\right\vert }{%
g\left( s-t\right) }>\sqrt{1+\varepsilon }\right)  \notag \\
& \leq \sum_{k=0}^{2^{n+1}-1}\sum_{l=0}^{\left[ \delta 2^{n+1}\right] +1}%
\mathbb{P}\left( \underset{\left\vert s-t\right\vert \leq \delta }{%
\sup_{(t,s)\in I_{k}\times I_{k+l}}}\frac{\left\vert
W_{s}^{n}-W_{t}^{n}\right\vert }{g\left( s-t\right) }>\sqrt{1+\varepsilon }%
\right) .  \label{nth bound (modulus)}
\end{align}%
Fix $k$ and consider different $l$ for the set 
\begin{equation*}
\left\{ \underset{\left\vert s-t\right\vert \leq \delta }{\sup_{(t,s)\in
I_{k}\times I_{k+l}}}\frac{\left\vert W_{s}^{n}-W_{t}^{n}\right\vert }{%
g\left( s-t\right) }>\sqrt{1+\varepsilon }\right\} .
\end{equation*}%
Note, if $l=0$ or $l=1,$ $\left\vert s-t\right\vert $ is not necessarily
bounded away from zero and must be treated with care. If $l>1,$ we must be
mindful of the shape of the underlying set.

If $l=0,$ $s$ and $t$ lie in the same dyadic interval $I_{j},$ for $j\leq
n+1.$ Thus, for each $j\leq n+1,$ there exists an integer $k_{j}$ such that $%
s,t\in \lbrack k_{j}2^{-j},(k_{j}+1)2^{-j})$. Moreover, 
\begin{equation*}
W_{s}^{n}-W_{t}^{n}=\left( s-t\right) \left(
X_{o}+\sum_{j=0}^{n}2^{j/2}(-1)^{\varepsilon _{j}\left( t\right)
}X_{j,k_{j}}\right) ,\ 
\end{equation*}%
where $\varepsilon _{j}\left( t\right) $ is the $j^{th}$ term in the binary
expansion of $t$ (and $s$). Hence, 
\begin{align}
& \mathbb{P}\left( \underset{\left\vert s-t\right\vert \leq \delta }{%
\sup_{s,t\in I_{k}}}\frac{\left\vert W_{s}^{n}-W_{t}^{n}\right\vert }{%
g\left( s-t\right) }>\sqrt{1+\varepsilon }\right)  \notag \\
& =\mathbb{P}\left( \underset{\left\vert s-t\right\vert \leq \delta }{%
\sup_{s,t\in I_{k}}}\frac{\left\vert s-t\right\vert \left\vert
X_{o}+\sum\limits_{j=0}^{n}2^{j/2}(-1)^{\varepsilon _{j}\left( t\right)
}X_{j,k_{j}}\right\vert }{\sqrt{2\left\vert s-t\right\vert L\frac{1}{%
\left\vert s-t\right\vert }}}>\sqrt{1+\varepsilon }\right)  \notag \\
& \leq \mathbb{P}\left( \left\vert N\left( 0,1\right) \right\vert >\sqrt{%
2\left( 1+\varepsilon \right) L\frac{1}{\delta _{n}}}\right) \leq \frac{%
\left( \delta _{n}\right) ^{1+\varepsilon }}{\sqrt{\pi L\frac{1}{\delta _{n}}%
}}.  \label{l=0}
\end{align}

Next if $l=1,$ the difference $s-t$ is no more than $2^{-n}$ but is not
necessarily bounded away from zero, so we consider two cases, $\delta
<2^{-n-1}$ and $\delta \geq 2^{-n-1}$.

If $\delta <2^{-n-1},$ again by Lemma \ref{mod calc lemma}, maximum is
achieved at one of the two points: $\left( \left( k+1\right) /2^{n+1},\text{ 
}\left( k+1\right) /2^{n+1}-\delta \right) $ or $\left( \left( k+1\right)
/2^{n+1}\text{ }-\delta ,\text{ }\left( k+1\right) /2^{n+1}\right) $
yielding 
\begin{equation}
\mathbb{P}\left( \underset{\left\vert s-t\right\vert \leq \delta }{%
\sup_{\left( t,s\right) \in I_{k}\times I_{k+1}}}\frac{\left\vert
W_{s}^{n}-W_{t}^{n}\right\vert }{g\left( s-t\right) }>\sqrt{1+\varepsilon }%
\right) \leq 2\frac{\delta ^{1+\varepsilon }}{\sqrt{\pi L\frac{1}{\delta }}}.
\label{l=1, delta small}
\end{equation}

If $\delta \geq 2^{-n-1},$ by Lemma \ref{mod calc lemma}, 
\begin{equation*}
\underset{\left\vert s-t\right\vert <2^{-n}\wedge \delta }{\sup_{(t,s)\in
I_{k}\times I_{k+1}}}\frac{\left\vert W_{s}^{n}-W_{t}^{n}\right\vert }{%
g\left( s-t\right) }=\underset{2^{-n-1}\leq \left\vert s-t\right\vert
<2^{-n}\wedge \delta }{\sup_{(t,s)\in I_{k}\times I_{k+1}}}\frac{\left\vert
W_{s}^{n}-W_{t}^{n}\right\vert }{g\left( s-t\right) }.
\end{equation*}
and maximum is achieved at one of the four points: \hspace{0.1in}%
\begin{equation*}
\begin{array}{cc}
(k/2^{n+1}, (k+1)/2^{n+1}), & \left( k/2^{n+1}, k/2^{n+1}+ \left(
2^{-n}\wedge \delta \right) \right), \\ 
{} & {} \\ 
\left( \left( k+2\right) /2^{n+1}-\text{ }\left( 2^{-n}\wedge \delta
\right), \left( k+2\right) /2^{n+1}\right), \ \text{or} & \left(
(k+1)/2^{n+1},(k+2)/2^{n+1}\right). \\ 
& 
\end{array}%
\end{equation*}

Thus 
\begin{align}
& \mathbb{P}\left( \underset{\left\vert s-t\right\vert \leq \delta }{%
\sup_{\left( t,s\right) \in I_{k}\times I_{k+1}}}\frac{\left\vert
W_{s}^{n}-W_{t}^{n}\right\vert }{g\left( s-t\right) }>\sqrt{1+\varepsilon }%
\right)  \notag \\
& \leq 2\left[ \frac{2^{-\left( n+1\right) \left( 1+\varepsilon \right) }}{%
\sqrt{\pi L2^{n+1}}}+\frac{2^{-n\left( 1+\varepsilon \right) }\wedge \delta
^{1+\varepsilon }}{\sqrt{\pi L\left( 2^{n}\vee \frac{1}{\delta }\right) }}%
\right] \leq 4\frac{2^{-n\left( 1+\varepsilon \right) }}{\sqrt{\pi L2^{n}}}.
\label{l=1, delta large}
\end{align}

Finally, if $l>1,$ we consider the three cases: $1<l\leq \left[
2^{n+1}\delta \right] -1,$ $l=\left[ 2^{n+1}\delta \right] ,$ and $l=\left[
2^{n+1}\delta \right] +1.$ Each implies a different underlying shape of the
set over which supremum is considered. They form rectangles, a pentagon, and
a triangle respectively. We employ Lemma \ref{mod calc lemma} in each
situation.

When $1<l\leq \left[ 2^{n+1}\delta \right] -1,$ the function 
\begin{equation*}
(t,s)\longmapsto \frac{\left\vert W_{s}^{n}-W_{t}^{n}\right\vert }{g\left(
s-t\right) }
\end{equation*}%
over the rectangle $I_{k}\times I_{k+l}$ achieves its maximum at one of the
four corner points, yielding upper bound%
\begin{align}
& \mathbb{P}\left( \underset{\left\vert t-s\right\vert \leq \delta }{%
\sup_{\left( t,s\right) \in I_{k}\times I_{k+l}}}\frac{\left\vert
W_{s}^{n}-W_{t}^{n}\right\vert }{g\left( s-t\right) }>\sqrt{1+\varepsilon }%
\right)  \notag \\
& \leq \frac{\left( \frac{l-1}{2^{n+1}}\right) ^{1+\varepsilon }}{\sqrt{\pi L%
\frac{2^{n+1}}{l-1}}}+2\frac{\left( \frac{l}{2^{n+1}}\right) ^{1+\varepsilon
}}{\sqrt{\pi L\frac{2^{n+1}}{l}}}+\frac{\left( \frac{l+1}{2^{n+1}}\right)
^{1+\varepsilon }}{\sqrt{\pi L\frac{2^{n+1}}{l+1}}}  \notag \\
& \leq \frac{1}{\sqrt{\pi L\frac{1}{\delta }}}\left( \left( \frac{l-1}{%
2^{n+1}}\right) ^{1+\varepsilon }+2\left( \frac{l}{2^{n+1}}\right)
^{1+\varepsilon }+\left( \frac{l+1}{2^{n+1}}\right) ^{1+\varepsilon }\right)
.  \label{l>1, rectangle set}
\end{align}

When $l=\left[ 2^{n+1}\delta \right] ,$ the set $I_{k}\times I_{k+l}\cap
\left\{ \left( t,s\right) :\left\vert s-t\right\vert \leq \delta \right\} $
is a pentagon; thus there are five extreme points yielding 
\begin{align}
& \mathbb{P}\left( \underset{\left\vert s-t\right\vert <\delta }{%
\sup_{\left( t,s\right) \in I_{k},I_{k+\left[ 2^{n+1}\delta \right] }}}\frac{%
\left\vert W_{s}^{n}-W_{t}^{n}\right\vert }{g\left( s-t\right) }>\sqrt{%
1+\varepsilon }\right)  \notag \\
& \leq 2\frac{\left( \frac{\left[ 2^{n+1}\delta \right] }{2^{n+1}}\right)
^{1+\varepsilon }}{\sqrt{\pi L\frac{2^{n+1}}{\left[ 2^{n+1}\delta \right] }}}%
+\frac{\left( \frac{\left[ 2^{n+1}\delta \right] -1}{2^{n+1}}\right)
^{1+\varepsilon }}{\sqrt{\pi L\frac{2^{n+1}}{\left[ 2^{n+1}\delta \right] -1}%
}}+2\frac{\delta ^{1+\varepsilon }}{\sqrt{\pi L\frac{1}{\delta }}}  \notag \\
& \leq \frac{1}{\sqrt{\pi L\frac{1}{\delta }}}\left( 2\left( \frac{\left[
2^{n+1}\delta \right] }{2^{n+1}}\right) ^{1+\varepsilon }+\left( \frac{\left[
2^{n+1}\delta \right] -1}{2^{n+1}}\right) ^{1+\varepsilon }+2\delta
^{1+\varepsilon }\right) .  \label{l>1, pentagon set}
\end{align}

When $l=\left[ 2^{n+1}\delta \right] +1$, the set $I_{k}\times I_{k+l}\cap
\left\{ \left( t,s\right) :\left\vert s-t\right\vert \leq \delta \right\} $
is a triangle; thus 
\begin{equation}
\mathbb{P}\left( \underset{\left\vert s-t\right\vert <\delta }{\sup_{\left(
t,s\right) \in I_{k}\times I_{k+\left[ 2^{n+1}\delta \right] +1}}}\frac{%
\left\vert W_{s}^{n}-W_{t}^{n}\right\vert }{g\left( s-t\right) }>\sqrt{%
1+\varepsilon }\right) \leq \frac{\left( \frac{\left[ 2^{n+1}\delta \right]
-1}{2^{n+1}}\right) ^{1+\varepsilon }+2\delta ^{1+\varepsilon }}{\sqrt{\pi L%
\frac{1}{\delta }}}.  \label{l>1, triangle set}
\end{equation}

We incorporate the upper bounds obtained from inequalities \ref{l=0} and \ref%
{l=1, delta small} into \ref{nth bound (modulus)} and see, for $0<\delta
<2^{-n-1},$ 
\begin{equation*}
\mathbb{P}\left( \underset{\left\vert s-t\right\vert \leq \delta }{%
\sup_{0\leq t<s\leq 1}}\frac{\left\vert W_{s}^{n}-W_{t}^{n}\right\vert }{%
g\left( s-t\right) }>\sqrt{1+\varepsilon }\right) \leq \frac{3\delta
^{\varepsilon }}{\sqrt{\pi L\frac{1}{\delta }}}.
\end{equation*}%
From inequalities \ref{l=0} and \ref{l=1, delta large}, for $\delta \geq
2^{-n-1},$ we have%
\begin{align*}
& \mathbb{P}\left( \underset{\left\vert s-t\right\vert \leq \delta }{%
\sup_{0\leq t<s\leq 1}}\frac{\left\vert W_{s}^{n}-W_{t}^{n}\right\vert }{%
g\left( s-t\right) }>\sqrt{1+\varepsilon }\right) \\
& \leq \frac{2^{-\varepsilon \left( n+1\right) }}{\sqrt{\pi L2^{n+1}}}+\frac{%
2^{-n\varepsilon +1}}{\sqrt{\pi L2^{n}}}+2^{n+1}\sum_{l=2}^{\left[ \delta
2^{n+1}\right] +1}\mathbb{P}\left( \underset{\left\vert s-t\right\vert \leq
\delta }{\sup_{(t,s)\in I_{k}\times I_{k+l}}}\frac{\left\vert
W_{s}^{n}-W_{t}^{n}\right\vert }{g\left( s-t\right) }>\sqrt{1+\varepsilon }%
\right) .
\end{align*}%
Using inequalities \ref{l>1, rectangle set}, \ref{l>1, pentagon set} and \ref%
{l>1, triangle set}, the sum in the previous inequality can be estimated. 
\begin{align*}
& \sqrt{\pi L\frac{1}{\delta }}\sum_{l=2}^{\left[ \delta 2^{n+1}\right] +1}%
\mathbb{P}\left( \underset{\left\vert s-t\right\vert \leq \delta }{%
\sup_{(t,s)\in I_{k}\times I_{k+l}}}\frac{\left\vert
W_{s}^{n}-W_{t}^{n}\right\vert }{g\left( s-t\right) }>\sqrt{1+\varepsilon }%
\right) \\
& \leq \sum_{l=2}^{\left[ 2^{n+1}\delta \right] -1}\left( \left( \frac{l-1}{%
2^{n+1}}\right) ^{1+\varepsilon }+2\left( \frac{l}{2^{n+1}}\right)
^{1+\varepsilon }+\left( \frac{l+1}{2^{n+1}}\right) ^{1+\varepsilon }\right)
+ \\
& \text{ \ \ \ \ \ }+2\left( \frac{\left[ 2^{n+1}\delta \right] }{2^{n+1}}%
\right) ^{1+\varepsilon }+2\left( \frac{\left[ 2^{n+1}\delta \right] -1}{%
2^{n+1}}\right) ^{1+\varepsilon }+4\delta ^{1+\varepsilon } \\
& \leq \left( \sum_{l=1}^{\left[ 2^{n+1}\delta \right] }\left( \frac{l}{%
2^{n+1}}\right) ^{1+\varepsilon }+\sum_{l=2}^{\left[ 2^{n+1}\delta \right]
}2\left( \frac{l}{2^{n+1}}\right) ^{1+\varepsilon }+\sum_{l=3}^{\left[
2^{n+1}\delta \right] }\left( \frac{l}{2^{n+1}}\right) ^{1+\varepsilon
}\right) +4\delta ^{1+\varepsilon }.
\end{align*}%
Hence, for $\delta \geq 2^{-n-1},$ 
\begin{align*}
& \mathbb{P}\left( \underset{\left\vert s-t\right\vert \leq \delta }{%
\sup_{0\leq t<s\leq 1}}\frac{\left\vert W_{s}^{n}-W_{t}^{n}\right\vert }{%
g\left( s-t\right) }>\sqrt{1+\varepsilon }\right) \\
& \leq \frac{2^{-\varepsilon \left( n+1\right) }}{\sqrt{\pi L\frac{1}{\delta 
}}}\left( 1+2^{3+\varepsilon }\sqrt{1+\frac{1}{n}}+2^{n+1}\left(
4\sum_{l=1}^{\left[ 2^{n+1}\delta \right] }\left( \frac{l}{2^{n+1}}\right)
^{1+\varepsilon }+4\left( \delta \right) ^{1+\varepsilon }\right) \right) \\
& \leq \frac{2^{-\varepsilon \left( n+1\right) }}{\sqrt{\pi L\frac{1}{\delta 
}}}K\left( \varepsilon ,\delta ,n\right)
\end{align*}%
where $K\left( \varepsilon ,\delta ,n\right) =1+9\left( 2\right)
^{\varepsilon }+2\left( 2^{n+1}\delta \right) ^{2+\varepsilon }+4\left(
2^{n+1}\delta \right) ^{1+\varepsilon }.$
\end{proof}

Now we establish an estimate for the distribution function of the global
maximal deviation of the ratio of the Brownian increment and the modulus of
continuity function for a fixed $\delta $. For monotonicity as explained in
Remark \ref{restriction on n} we insist that $\delta \leq 2^{-2}$ and for
future purposes, we will need $\delta \leq 2^{-5}.$ Without this
restriction, the constant, $K_{1}\left( \varepsilon \right) $ in the theorem
below would be greater.

\begin{theorem}
\label{deviations for fixed delta (modulus)}Let $0<\delta \leq 2^{-5}$ and $%
\varepsilon >0$. Then 
\begin{equation}
\mathbb{P}\left( \underset{\left\vert s-t\right\vert \leq \delta }{%
\sup_{0\leq t<s\leq 1}}\frac{\left\vert W_{s}-W_{t}\right\vert }{g(\delta
)r\left( \delta \right) }>\sqrt{1+\varepsilon }\right) \leq K_{1}\left(
\varepsilon \right) \delta ^{\varepsilon }\left( L\frac{1}{\delta }\right) ^{%
\frac{3}{2}},  \notag
\end{equation}%
where 
\begin{equation*}
r\left( \delta \right) =\left( 1+\frac{2.65}{\sqrt{L\frac{1}{\delta }}}%
\right)
\end{equation*}%
and 
\begin{equation*}
K_{1}\left( \varepsilon \right) =27.95+\frac{0.11}{\varepsilon }\mathbf{1}%
_{(0,1)}(\varepsilon )
\end{equation*}
\end{theorem}

\begin{proof}
Let $\varepsilon >0$ and $0<\delta \leq 2^{-5}$ and $n\geq 8$, so that $%
(n+1)2^{-n-1}<\delta \leq n2^{-n}$.

The proof is completed in two steps. First we estimate the size of the set 
\begin{equation*}
A_{\varepsilon ,\delta }=\left\{ \underset{\left\vert s-t\right\vert \leq
n2^{-n}}{\sup_{0\leq t<s\leq 1}}\frac{\left\vert
W_{s}^{n}-W_{t}^{n}\right\vert }{g(t-s)}\leq \sqrt{1+\varepsilon }\right\}
\cap \left\{ \underset{0\leq k\leq 2^{j}-1}{\max_{j\geq n+1}}\frac{%
\left\vert X_{j,k}\right\vert }{\sqrt{L2^{j}}}\leq \sqrt{2\left(
1+\varepsilon \right) }\right\}
\end{equation*}%
using both Theorem \ref{mod theorem nth partial sum} and Lemma \ref{tail
estimate}. Then we show that on $A_{\varepsilon ,\delta }$, 
\begin{equation*}
\left\vert W_{s}-W_{t}\right\vert \leq g\left( \delta \right) r\left( \delta
\right) \sqrt{1+\varepsilon }
\end{equation*}%
for all $\left\vert s-t\right\vert \leq \delta .$

By Theorem \ref{mod theorem nth partial sum} and Lemma \ref{tail estimate},
we have 
\begin{align*}
\mathbb{P}\left( A_{\varepsilon ,\delta }^{c}\right) & \leq \frac{%
2^{-\varepsilon \left( n+1\right) }}{\sqrt{\pi L\frac{1}{\delta }}}\left(
K\left( \varepsilon ,\delta ,n\right) +\frac{1}{\left( 1-2^{-\varepsilon
}\right) }\right) \\
& \leq \frac{2^{-\varepsilon \left( n+1\right) }}{\sqrt{\pi L\frac{1}{\delta 
}}}\left( 1+9\left( 2\right) ^{\varepsilon }+2\left( 2^{n+1}\delta \right)
^{2+\varepsilon }+4\left( 2^{n+1}\delta \right) ^{1+\varepsilon }+\frac{1}{%
\left( 1-2^{-\varepsilon }\right) }\right) \\
& \leq \frac{\delta ^{\varepsilon }}{\sqrt{\pi L\frac{1}{\delta }}}\left(
4\left( 2^{n+1}\delta \right) +2\left( 2^{n+1}\delta \right) ^{2}+\frac{%
1+9\left( 2\right) ^{\varepsilon }+\frac{1}{\left( 1-2^{-\varepsilon
}\right) }}{\left( 2^{n+1}\delta \right) ^{\varepsilon }}\right) \\
& \leq \,c\delta ^{\varepsilon }\left( L\frac{1}{\delta }\right) ^{\frac{3}{2%
}}\left( 4\left( 2^{4}\right) +2\left( 2^{4}\right) ^{2}+\frac{1+9\left(
2\right) ^{\varepsilon }+\frac{1}{\left( 1-2^{-\varepsilon }\right) }}{%
\left( n+1\right) ^{\varepsilon }}\right)
\end{align*}%
since $L\frac{1}{\delta }\geq L2^{n}$ and where $c=\left( \sqrt{\pi }\left(
L2^{5}\right) ^{2}\right) ^{-1}.$

To eliminate the dependency on $n$ in our above estimate of $\mathbb{P}%
\left( A_{\varepsilon ,\delta }^{c}\right) ,$ we consider $0<\varepsilon <1$
and $\varepsilon \geq 1.$ If $0<\varepsilon <1$, then 
\begin{equation*}
\mathbb{P}\left( A_{\varepsilon ,\delta }^{c}\right) \leq c\delta
^{\varepsilon }\left( L\frac{1}{\delta }\right) ^{\frac{3}{2}}\left(
2^{6}+2^{9}+10+\frac{2\varepsilon ^{-1}}{L2\left( 2-L2\right) }\right)
\end{equation*}%
and if $\varepsilon \geq 1,$ then 
\begin{equation*}
\mathbb{P}\left( A_{\varepsilon ,\delta }^{c}\right) \leq c\delta
^{\varepsilon }\left( L\frac{1}{\delta }\right) ^{\frac{3}{2}}\left(
2^{6}+2^{9}+\frac{1}{3}+9\left( \frac{2}{9}\right) ^{\varepsilon }\right) .
\end{equation*}%
Thus for all $\varepsilon >0$, we have 
\begin{equation*}
\mathbb{P}\left( A_{\varepsilon ,\delta }^{c}\right) \leq \left( 27.95+\frac{%
0.11}{\varepsilon }\mathbf{1}_{(0,1)}(\varepsilon )\right) \delta
^{\varepsilon }\left( L\frac{1}{\delta }\right) ^{\frac{3}{2}},
\end{equation*}

Next we estimate $\left\vert W_{s}-W_{t}\right\vert $ on the set $%
A_{\varepsilon ,\delta }.$ By Theorem \ref{mod theorem nth partial sum} and
Lemma \ref{tail estimate} and recalling that $n\geq 8$, we note that $W$
restricted to $A_{\varepsilon ,\delta }$ yields 
\begin{align}
\left\vert W_{s}-W_{t}\right\vert & \leq \left\vert
W_{s}^{n}-W_{t}^{n}\right\vert +\left\vert \sum_{j=n+1}^{\infty
}2^{-j/2}\sum_{k=0}^{2^{j}-1}\left[ \Lambda _{j,k}\left( t\right) -\Lambda
_{j,k}\left( s\right) \right] X_{j,k}\right\vert  \notag \\
& \leq g\left( s-t\right) \sqrt{1+\varepsilon }+2\sum_{j=n+1}^{\infty }\frac{%
2^{-j/2}\sqrt{j}}{2}\underset{0\leq k<2^{j}-1}{\max_{j\geq n+1}}\frac{%
\left\vert X_{j,k}\right\vert }{\sqrt{j}}  \notag \\
& \leq \sqrt{1+\varepsilon }\left( g\left( s-t\right) +\sqrt{2L2}%
\sum_{j=n+1}^{\infty }2^{-j/2}\sqrt{j}\right)  \notag \\
& \leq \sqrt{1+\varepsilon }\left( g\left( s-t\right) +\sqrt{2L2}\sqrt{\frac{%
n+1}{2^{n+1}}}\sum_{j=0}^{\infty }2^{-j/2}\sqrt{1+\frac{j}{9}}\right)  \notag
\\
& \leq \sqrt{1+\varepsilon }\left( g(s-t)+2.65\sqrt{2\frac{n+1}{2^{n+1}}}%
\right) .  \label{tail and truncated part}
\end{align}%
since 
\begin{equation*}
\sqrt{L2}\sum_{j=0}^{\infty }2^{-j/2}\sqrt{1+\frac{j}{9}}\leq 2.65.
\end{equation*}%
Recall $\frac{n+1}{2^{n+1}}<\delta $ and that the inequality \ref{tail and
truncated part} holds for all $\left\vert s-t\right\vert \leq \delta $, we
have 
\begin{equation*}
\sup_{\left\vert s-t\right\vert \leq \delta }\frac{\left\vert
W_{t}-W_{s}\right\vert }{g\left( \delta \right) }\leq \sqrt{1+\varepsilon }%
\left( 1+\frac{2.65}{\sqrt{L\frac{1}{\delta }}}\right)
\end{equation*}%
on the set $A_{\varepsilon ,\delta }$ whose probability is greater than 
\begin{equation*}
1-K_{1}\left( \varepsilon \right) \delta ^{\varepsilon }\left( L\frac{1}{%
\delta }\right) ^{\frac{3}{2}}.
\end{equation*}
\end{proof}

For practical purposes, results uniformly over $\delta $ are of interest,
thus the results of the previous theorem are not as desirable. Moreover, the
function 
\begin{equation*}
\delta \rightarrow \sup_{\left\vert s-t\right\vert \leq \delta }\frac{%
\left\vert W_{t}-W_{s}\right\vert }{g\left( \delta \right) }
\end{equation*}%
is not necessarily monotonic which make establishing uniform results
challenging. The theorem below addresses this need and challenge. We should
note that its proof is similar to the proof of Theorem \ref{deviations for
fixed delta (modulus)}, and it yields the same rate in $\delta .$
Additionally, it may seem at first glance, the expressions $K_{1}$ and $%
K_{2} $ in the these two theorems may look different. However, they really
only differ near $\varepsilon $ zero. $K_{1},$ above behaves as $\varepsilon
^{-1} $ near zero while $K_{2}$ behaves as $\varepsilon ^{-3}$. Moreover
their corresponding multipliers are extremely different as well with the
coefficient of the $\varepsilon ^{-1}$ being a hundredth of $\varepsilon
^{-3}$. But as $\varepsilon $ moves away from zero, the two behave basically
the same.

\begin{theorem}
\label{uniform mod theorem}Let $0<\delta _{o}\leq 2^{-5}$ and $\varepsilon
>0 $. Then 
\begin{equation*}
\mathbb{P}\left( \sup_{\delta \leq \delta _{o}}\underset{\left\vert
s-t\right\vert \leq \delta }{\sup_{0\leq t<s\leq 1}}\frac{\left\vert
W_{s}-W_{t}\right\vert }{g(\delta )r\left( \delta \right) }>\sqrt{%
1+\varepsilon }\right) \leq K_{2}\left( \varepsilon \right) \delta
_{o}^{\varepsilon }\left( L\frac{1}{\delta _{o}}\right) ^{\frac{3}{2}}\text{,%
}
\end{equation*}%
where 
\begin{equation*}
\frac{9.57}{\varepsilon ^{3}}\mathbf{1}_{(0,2a]}(\varepsilon )+\left( \frac{%
14.59}{\varepsilon }+9.9\right) \mathbf{1}_{(2a,\infty ]}(\varepsilon
)+24.05,
\end{equation*}%
where $a=\left( 8L2-1\right) ^{-1}$ and $r\left( \delta \right) $ is as in
Theorem \ref{deviations for fixed delta (modulus)}.
\end{theorem}

\begin{proof}
Let $\delta _{o}\leq 2^{-5}$ and $n$ be such that $\left( n+1\right)
2^{-n-1}<\delta _{o}\leq n2^{-n.}.$ Our choice of $\delta _{o}$ forces $%
n\geq 8.$ Set 
\begin{eqnarray*}
A_{\varepsilon ,\delta _{o}} &=&\left\{ \underset{\left( n+1\right)
2^{-n-1}<\left\vert s-t\right\vert \leq \delta _{o}}{\sup_{0\leq t<s\leq 1}}%
\frac{\left\vert W_{s}^{n}-W_{t}^{n}\right\vert }{g(t-s)}\leq \sqrt{%
1+\varepsilon }\right\}  \\
&&\cap \left\{ \sup_{m\geq n+1}\underset{\left\vert s-t\right\vert \leq
m2^{-m}}{\sup_{0\leq t<s\leq 1}}\frac{\left\vert
W_{s}^{m}-W_{t}^{m}\right\vert }{g(t-s)}\leq \sqrt{1+\varepsilon }\right\} 
\\
&&\cap \left\{ \underset{0\leq k\leq 2^{j}-1}{\max_{j>n}}\frac{\left\vert
X_{j,k}\right\vert }{\sqrt{L2^{j}}}\leq \sqrt{2\left( 1+\varepsilon \right) }%
\right\} 
\end{eqnarray*}%
and define $\delta _{m}=m2^{-m}.$ Define $S_{1,}$ $S_{2,}$ and $S_{2}$ so
that $\mathbb{P}\left( A_{\varepsilon ,\delta _{o}}^{c}\right) \leq
S_{1}+S_{2}+S_{3}$ and the following holds. The first term, $S_{1},$ is
derived from the first set used to create $A_{\varepsilon ,\delta _{o}}.$
The upper bound of its compliment is determined in the same fashion as the
proof of Theorem \ref{mod theorem nth partial sum} for $\delta >2^{-n-1}$.
Specifically, 
\begin{equation*}
S_{1}\leq \frac{1}{\sqrt{\pi L\frac{1}{\delta _{o}}}}\left( 2^{n+1}\left(
4\sum_{l=n+2}^{\left[ 2^{n+1}\delta _{o}\right] }\left( \frac{l}{2^{n+1}}%
\right) ^{1+\varepsilon }+4\left( \delta _{o}\right) ^{1+\varepsilon
}\right) \right) 
\end{equation*}%
since $\left( n+1\right) 2^{-n-1}<\delta _{o}.$ we could approximate this
sum by an integral which would have resulted in a constant divided by $%
2+\varepsilon .$ This would be beneficial in practice when $\varepsilon $ is
large. However the constants would more than double. Instead, we replace the
sum with its greatest summand and see that 
\begin{eqnarray}
S_{1} &\leq &\frac{\left( L\frac{1}{\delta _{n}}\right) ^{2}}{\sqrt{\pi L%
\frac{1}{\delta _{o}}}}\frac{4n\left( 2^{n+1}\delta _{o}\right)
^{1+\varepsilon }}{\left( L\frac{1}{\delta _{n}}\right)
^{2}2^{(n+1)\varepsilon }}  \notag \\
&\leq &\frac{\left( L\frac{1}{\delta _{o}}\right) ^{\frac{3}{2}}}{\sqrt{\pi }%
}\frac{8n^{2}\delta _{0}^{\varepsilon }}{\left( nL2-Ln\right) ^{2}}  \notag
\\
&\leq &\frac{\delta _{o}^{\varepsilon }\left( L\frac{1}{\delta _{0}}\right)
^{\frac{3}{2}}}{\sqrt{\pi }}\frac{8}{\left( L2-\frac{Ln}{n}\right) ^{2}} 
\notag \\
&\leq &\delta _{o}^{\varepsilon }\left( L\frac{1}{\delta _{o}}\right) ^{%
\frac{3}{2}}\frac{8}{\sqrt{\pi }\left( L2-\frac{L8}{8}\right) ^{2}}%
<24.05\delta _{o}^{\varepsilon }\left( L\frac{1}{\delta _{o}}\right) ^{\frac{%
3}{2}}.  \label{s1epsilonany}
\end{eqnarray}

$S_{2}$ and $S_{3}$ are derived from the last two sets used to create $%
A_{\varepsilon ,\delta _{o}}.$ Specifically, according to Theorem \ref{mod
theorem nth partial sum} and Lemma \ref{tail estimate}, 
\begin{eqnarray*}
&&P\left( A_{\varepsilon ,\delta _{o}}\backslash \left\{ \underset{%
\left\vert s-t\right\vert \leq \delta _{o}}{\sup_{0\leq t<s\leq 1}}\frac{%
\left\vert W_{s}^{n}-W_{t}^{n}\right\vert }{g(t-s)}\leq \sqrt{1+\varepsilon }%
\right\} \right) \\
&\leq &\sum_{m=n+1}^{\infty }\frac{2^{-\varepsilon \left( m+1\right)
}K\left( \varepsilon ,\delta _{m},m\right) }{\sqrt{\pi L\frac{1}{\delta _{m}}%
}}+\frac{2^{-\varepsilon \left( n+1\right) }}{\left( 1-2^{-\varepsilon
}\right) \sqrt{\pi L2^{n+1}}} \\
&\leq &\sum_{m=n+1}^{\infty }\frac{2^{-\varepsilon \left( m+1\right) }}{%
\sqrt{\pi L\frac{1}{\delta _{m}}}}\left( 1+9\left( 2\right) ^{\varepsilon
}+2\left( 2^{m+1}\delta _{m}\right) ^{2+\varepsilon }+4\left( 2^{m+1}\delta
_{m}\right) ^{1+\varepsilon }\right) \\
&&+\frac{2^{-\varepsilon \left( n+1\right) }}{\left( 1-2^{-\varepsilon
}\right) \sqrt{\pi L2^{n+1}}} \\
&\leq &\sum_{m=n+1}^{\infty }\frac{2^{-\varepsilon \left( m+1\right) }}{%
\sqrt{\pi L\frac{1}{\delta _{m}}}}\left( 2\left( 2^{m+1}\delta _{m}\right)
^{2+\varepsilon }+4\left( 2^{m+1}\delta _{m}\right) ^{1+\varepsilon }\right)
+ \\
&&\sum_{m=n+1}^{\infty }\frac{2^{-\varepsilon \left( m+1\right) }}{\sqrt{\pi
L\frac{1}{\delta _{m}}}}\left( 1+9\left( 2\right) ^{\varepsilon }\right) +%
\frac{2^{-\varepsilon \left( n+1\right) }}{\left( 1-2^{-\varepsilon }\right) 
\sqrt{\pi L2^{n+1}}} \\
&=&S_{2}+S_{3}
\end{eqnarray*}

Consider $S_{2}$ by looking at the sum below for $k=1,2.$ 
\begin{equation*}
\sum_{m=n+1}^{\infty }2^{-\varepsilon m}m^{k+\varepsilon }\leq
2^{-\varepsilon \left( n+1\right) }\left( n+1\right) ^{k+\varepsilon
}\sum_{m=0}^{\infty }2^{-\varepsilon m}\left( 1+\frac{m}{8}\right)
^{k+\varepsilon }.
\end{equation*}%
Define 
\begin{equation}
I_{k}(\varepsilon )=\sum_{m=0}^{\infty }2^{-\varepsilon m}\left( 1+\frac{m}{8%
}\right) ^{k+\varepsilon }  \label{series}
\end{equation}%
for $k=1,2.$ Although we could use $I_{1}(\varepsilon )\leq
I_{2}(\varepsilon ),$ we estimate both, $I_{1}(\varepsilon )$ and $%
I_{2}(\varepsilon )$ to obtain better constants. The function $%
f_{k}:[0,\infty )\rightarrow \mathbb{R}$ defined by 
\begin{equation}
f_{k}(x)=\left( 1+\frac{x}{8}\right) ^{\varepsilon +k}2^{-\varepsilon x}
\label{series terms}
\end{equation}%
has only one maximum which is achieved at 
\begin{equation*}
x_{o}=8\left( \frac{k+\varepsilon }{8\varepsilon L2}-1\right)
\end{equation*}%
provided $\varepsilon \leq \frac{k}{8L2-1};$ otherwise, the maximum appears
at $x=0$. Therefore when $\varepsilon \leq \frac{k}{8L2-1}=ak,$ where $%
a=\left( 8L2-1\right) ^{-1},$ we have 
\begin{eqnarray*}
I_{k}(\varepsilon ) &\leq &\int_{0}^{\left\lfloor x_{o}\right\rfloor }\left(
1+\frac{x}{8}\right) ^{k+\varepsilon }2^{-\varepsilon x}dx+\int_{\left\lceil
x_{o}\right\rceil }^{\infty }\left( 1+\frac{x-1}{8}\right) ^{k+\varepsilon
}2^{-\varepsilon \left( x-1\right) }dx+f_{k}(x_{o}) \\
&\leq &\int_{0}^{\infty }\left( 1+\frac{x}{8}\right) ^{k+\varepsilon
}2^{-\varepsilon x}dx+f_{k}(x_{0}).
\end{eqnarray*}%
Otherwise, when $\varepsilon >ak,$%
\begin{equation*}
I_{k}(\varepsilon )\leq \int_{0}^{\infty }\left( 1+\frac{x}{8}\right)
^{k+\varepsilon }2^{-\varepsilon x}dx+1.
\end{equation*}%
Thus we can combine both cases and create the upper bound 
\begin{equation*}
I_{k}(\varepsilon )\leq \int_{0}^{\infty }\left( 1+\frac{x}{8}\right)
^{k+\varepsilon }2^{-\varepsilon x}dx+f_{k}(x_{0})\mathbf{1}%
_{(0,a]}(\varepsilon )+\mathbf{1}_{(a,\infty )}(\varepsilon ).
\end{equation*}

One remark regarding $f_{k}\left( x_{o}\right) $ which we will need later on
for $\varepsilon \leq ak$ is that 
\begin{eqnarray*}
f_{k}(x_{o}) &=&\left( \frac{k+\varepsilon }{8\varepsilon L2}\right)
^{k+\varepsilon }2^{-8\left( \frac{k+\varepsilon }{8L2}-\varepsilon \right) }
\\
&\leq &\left( \frac{1}{\varepsilon }\right) ^{k+\varepsilon }\left( \frac{%
k+ak}{8L2}\right) ^{k+ak},
\end{eqnarray*}%
and thus, 
\begin{equation*}
f_{1}(x_{o})=\frac{0.16}{\varepsilon ^{1+\varepsilon }}\text{ and }%
f_{2}(x_{o})=\frac{0.14}{\varepsilon ^{2+\varepsilon }}.
\end{equation*}

In order to simplify the integral that appears in $I_{k}\left( \varepsilon
\right) ,$, we substitute $z=8\varepsilon L2\left( 1+\frac{x}{8}\right) .$
Then 
\begin{eqnarray}
&&\int_{0}^{\infty }\left( 1+\frac{x}{8}\right) ^{k+\varepsilon
}2^{-\varepsilon x}dx  \notag \\
&=&\frac{2^{8\varepsilon }}{\varepsilon ^{k+\varepsilon +1}\left( 8L2\right)
^{k+\varepsilon }\left( L2\right) }\int_{8\varepsilon L2}^{\infty
}z^{k+\varepsilon }e^{-z}dz.  \label{Iksubstituteintegral}
\end{eqnarray}

Since integral \ref{Iksubstituteintegral} is not easily integrated, we
employ a few tricks. Note 
\begin{equation}
z^{k+\varepsilon }e^{-z}=z^{k+\varepsilon }e^{-\frac{z}{b}}e^{-\left( 1-%
\frac{1}{b}\right) z}  \label{trick}
\end{equation}%
for all $b\neq 0.$ The idea is that we will replace the first two factors of %
\ref{trick} with their maximums, and integrate the last factor of \ref{trick}%
. The function 
\begin{equation*}
z\rightharpoonup z^{k+\varepsilon }e^{-\frac{z}{b}}
\end{equation*}%
attains maximum when $z=b\left( k+\varepsilon \right) $ for any $b>0$ and
the last factor is integrable if $b>1.$ The best upper bound would result by
choosing the minimal $b$ however, our choice of $b=16L2/\left( 8L2+1\right) $
produces constants which are easy to manipulate and not far from the minimal
constants..

Now we compute $I_{1}\left( \varepsilon \right) .$%
\begin{equation*}
I_{1}\left( \varepsilon \right) \leq \frac{1}{\varepsilon ^{\varepsilon
+2}\left( L2\right) ^{\varepsilon +2}2^{3-5\varepsilon }}\left(
\int_{8\varepsilon L2}^{\infty }z^{1+\varepsilon }e^{-z}dz\right)
+f_{k}(x_{0})\mathbf{1}_{(0,a]}(\varepsilon )+\mathbf{1}_{(a,\infty
)}(\varepsilon )
\end{equation*}%
Since we eventually combine results with $I_{2}$, we only consider the cases 
$\varepsilon \leq 2a$ and $\varepsilon >2a$ instead of $\varepsilon \leq a$
and $\varepsilon >a$. When $\varepsilon \leq 2a.$ 
\begin{eqnarray*}
I_{1}\left( \varepsilon \right) &\leq &\frac{1}{\varepsilon ^{\varepsilon
+2}\left( L2\right) ^{\varepsilon +2}2^{3-5\varepsilon }}\left(
\int_{8\varepsilon L2}^{\infty }z^{1+\varepsilon }e^{-z}dz\right) +\max
\left( f_{1}(x_{o}),1\right) \\
&\leq &\frac{1}{\varepsilon ^{\varepsilon +2}}\left( \left( \frac{32}{\left(
8L2+1\right) e^{1}}\right) ^{1+\varepsilon }\left( 1+\varepsilon \right)
^{1+\varepsilon }\frac{2^{\frac{\varepsilon }{2L2}}}{8L2-1}+\max \left(
0.16\varepsilon ,\varepsilon ^{\varepsilon +2}\right) \right) \\
&\leq &\frac{1.15}{\varepsilon ^{\varepsilon +2}}
\end{eqnarray*}%
It should be clear that the second factor in the second to the last line
above is increasing in $\varepsilon $; thus, the final result is obtained by
substituting $2a$ for $\varepsilon $.

For $\varepsilon >2a,$ computation is much easier since the maximum of the
function 
\begin{equation*}
z\rightarrow z^{1+\varepsilon }e^{-\frac{z}{b}}
\end{equation*}%
occurs before the lower limit of integration. Therefore, 
\begin{eqnarray*}
I_{1}\left( \varepsilon \right) &\leq &\frac{1}{\varepsilon ^{\varepsilon
+2}\left( L2\right) ^{\varepsilon +2}2^{3-5\varepsilon }}\left(
\int_{8\varepsilon L2}^{\infty }z^{1+\varepsilon }e^{-z}dz\right) +1 \\
&\leq &\frac{\left( 8\varepsilon L2\right) ^{1+\varepsilon }e^{-\frac{%
8\varepsilon L2}{b}}}{\varepsilon ^{\varepsilon +2}\left( L2\right)
^{\varepsilon +2}2^{3-5\varepsilon }}\int_{8\varepsilon L2}^{\infty
}e^{-\left( 1-\frac{1}{b}\right) z}dz+1 \\
&\leq &\frac{1}{\varepsilon L2}+1.
\end{eqnarray*}

For $k=2,$ we repeat the process. For $\varepsilon \leq 2a,$ 
\begin{eqnarray*}
I_{2}\left( \varepsilon \right) &\leq &\frac{1}{\varepsilon ^{\varepsilon
+3}\left( L2\right) ^{\varepsilon +3}2^{6-5\varepsilon }}\left(
\int_{8\varepsilon L2}^{\infty }z^{2+\varepsilon }e^{-z}dz\right)
+f_{2}(x_{0}) \\
&\leq &\frac{0.70}{\varepsilon ^{\varepsilon +3}}.
\end{eqnarray*}

And $\varepsilon >2a,$%
\begin{eqnarray*}
I_{2}\left( \varepsilon \right) &\leq &\frac{1}{\varepsilon ^{\varepsilon
+3}\left( L2\right) ^{\varepsilon +3}2^{6-5\varepsilon }}\left(
\int_{8\varepsilon L2}^{\infty }z^{2+\varepsilon }e^{-z}dz\right) +1 \\
&\leq &\frac{1}{\varepsilon L2}+1.
\end{eqnarray*}

We are now ready to bound $S_{2}.$%
\begin{eqnarray*}
S_{2} &=&\frac{4\sum_{m=n+1}^{\infty }2^{-\varepsilon m}m^{1+\varepsilon
}+2\sum_{m=n+1}^{\infty }2^{-\varepsilon m}m^{2+\varepsilon }}{\sqrt{\pi L%
\frac{1}{\delta _{n+1}}}} \\
&\leq &\frac{4\left( 2^{-\varepsilon \left( n+1\right) }\left( n+1\right)
^{1+\varepsilon }\right) I_{1}\left( \varepsilon \right) +2\left(
2^{-\varepsilon \left( n+1\right) }\left( n+1\right) ^{2+\varepsilon
}\right) I_{2}\left( \varepsilon \right) }{\sqrt{\pi L\frac{1}{\delta _{n+1}}%
}} \\
&\leq &\delta _{o}^{\varepsilon }\left( L\frac{1}{\delta _{o}}\right) ^{%
\frac{3}{2}}\left( \frac{4\left( n+1\right) I_{1}\left( \varepsilon \right)
+2\left( n+1\right) ^{2}I_{2}\left( \varepsilon \right) }{\sqrt{\pi L\frac{1%
}{\delta _{n+1}}}\left( L\frac{1}{\delta _{n}}\right) ^{\frac{3}{2}}}\right)
.
\end{eqnarray*}%
We bound by substituting and get 
\begin{equation*}
S_{2}\leq \delta _{o}^{\varepsilon }\left( L\frac{1}{\delta _{o}}\right) ^{%
\frac{3}{2}}\left( \frac{\frac{4I_{1}\left( \varepsilon \right) }{9}%
+2I_{2}\left( \varepsilon \right) }{\sqrt{\pi }\left( L2-\frac{1+L8}{9}%
\right) ^{\frac{3}{2}}\left( \ln 2-\frac{\ln 9}{9}\right) ^{\frac{1}{2}}}%
\right)
\end{equation*}%
We bring in results for $I_{1}$ and $I_{2}$ and see for $\varepsilon \leq 2a$
\begin{eqnarray}
S_{2} &\leq &\delta _{o}^{\varepsilon }\left( L\frac{1}{\delta _{o}}\right)
^{\frac{3}{2}}\left( \frac{\frac{4}{9}\left( \frac{1.15}{\varepsilon
^{\varepsilon +2}}\right) +2\left( \frac{0.70}{\varepsilon ^{\varepsilon +3}}%
\right) }{\sqrt{\pi }\left( L2-\frac{1+L8}{9}\right) ^{\frac{3}{2}}\left( L2-%
\frac{L9}{9}\right) ^{\frac{1}{2}}}\right)  \notag \\
&\leq &9.51\frac{\delta _{o}^{\varepsilon }\left( L\frac{1}{\delta _{o}}%
\right) ^{\frac{3}{2}}}{\varepsilon ^{3}}  \label{s2epsilonsmall}
\end{eqnarray}%
And if $\varepsilon >2a$%
\begin{eqnarray}
S_{2} &\leq &\delta _{o}^{\varepsilon }\left( L\frac{1}{\delta _{o}}\right)
^{\frac{3}{2}}\left( \frac{\frac{4I_{1}\left( \varepsilon \right) }{9}%
+2I_{2}\left( \varepsilon \right) }{\sqrt{\pi }\left( L2-\frac{1+L8}{9}%
\right) ^{\frac{3}{2}}\left( L2-\frac{L9}{9}\right) ^{\frac{1}{2}}}\right) 
\notag \\
&\leq &9.90\delta _{o}^{\varepsilon }\left( L\frac{1}{\delta _{o}}\right) ^{%
\frac{3}{2}}\left( \frac{1}{\varepsilon L2}+1\right) .
\label{s2epsilonlarge}
\end{eqnarray}%
Next we move on to our upper bound for $S_{3.}.$%
\begin{eqnarray*}
S_{3} &=&\frac{2^{-\varepsilon \left( n+2\right) }}{\sqrt{\pi L\frac{1}{%
\delta _{n+1}}}}\frac{\left( 1+9\left( 2\right) ^{\varepsilon }\right) }{%
1-2^{-\varepsilon }}+\frac{2^{-\varepsilon \left( n+1\right) }}{\sqrt{\pi
L2^{n+1}}\left( 1-2^{-\varepsilon }\right) } \\
&\leq &\frac{\delta _{o}^{\varepsilon }\left( L\frac{1}{\delta _{o}}\right)
^{\frac{3}{2}}}{\left( n+1\right) ^{\varepsilon }\left( 1-2^{-\varepsilon
}\right) \left( L\frac{1}{\delta _{n}}\right) ^{\frac{3}{2}}}\left( \frac{%
2^{-\varepsilon }+9}{\sqrt{\pi L\frac{1}{\delta _{n+1}}}}+\frac{1}{\sqrt{\pi
L2^{n+1}}}\right) \\
&\leq &\frac{\delta _{o}^{\varepsilon }\left( L\frac{1}{\delta _{o}}\right)
^{\frac{3}{2}}}{9^{\varepsilon }\left( 1-2^{-\varepsilon }\right) \left( L%
\frac{2^{8}}{8}\right) ^{\frac{3}{2}}\sqrt{\pi }}\left( \frac{%
2^{-\varepsilon }+9}{\sqrt{L\frac{2^{9}}{9}}}+\frac{1}{\sqrt{L2^{9}}}\right)
\end{eqnarray*}%
For consistency, we consider the cases $\varepsilon \leq 2a$ and $%
\varepsilon >2a.$ For $\varepsilon \leq 2a,$ 
\begin{eqnarray*}
S_{3} &\leq &\frac{\delta _{o}^{\varepsilon }\left( L\frac{1}{\delta _{o}}%
\right) ^{\frac{3}{2}}}{9^{\varepsilon }\left( \varepsilon L2-\frac{\left(
\varepsilon L2\right) ^{2}}{2}\right) \left( L\frac{2^{8}}{8}\right) ^{\frac{%
3}{2}}\sqrt{\pi }}\left( \frac{2^{-\varepsilon }+9}{\sqrt{L\frac{2^{9}}{9}}}+%
\frac{1}{\sqrt{L2^{9}}}\right) \\
&\leq &\frac{\delta _{o}^{\varepsilon }\left( L\frac{1}{\delta _{o}}\right)
^{\frac{3}{2}}}{\varepsilon ^{3}}\frac{\varepsilon ^{2}}{9^{\varepsilon
}\left( L2-\frac{\varepsilon \left( L2\right) ^{2}}{2}\right) \left( L\frac{%
2^{8}}{8}\right) ^{\frac{3}{2}}\sqrt{\pi }}\left( \frac{10}{\sqrt{L\frac{%
2^{9}}{9}}}+\frac{1}{\sqrt{L2^{9}}}\right)
\end{eqnarray*}%
Note that the function 
\begin{equation*}
\varepsilon \rightarrow \frac{\varepsilon ^{2}}{9^{\varepsilon }\left( L2-%
\frac{\varepsilon \left( L2\right) ^{2}}{2}\right) }
\end{equation*}%
increases on the set $\left( 0,2a\right] .$ Thus%
\begin{equation}
S_{3}\leq 0.06\frac{\delta _{o}^{\varepsilon }\left( L\frac{1}{\delta _{o}}%
\right) ^{\frac{3}{2}}}{\varepsilon ^{3}}  \label{s3epsilonsmall}
\end{equation}%
on that set. And when $\varepsilon >2a,$ we have 
\begin{equation}
S_{3}\leq 0.30\frac{\delta _{o}^{\varepsilon }\left( L\frac{1}{\delta _{o}}%
\right) ^{\frac{3}{2}}}{\varepsilon }  \label{s3epsilonlarge}
\end{equation}

Using the bounds \ref{s1epsilonany}, \ref{s2epsilonlarge}, \ref%
{s2epsilonsmall}, \ref{s3epsilonlarge}, and \ref{s3epsilonsmall}, and
considering the sets, determined by $a,$ in which $\varepsilon $ may lie, we
estimate 
\begin{equation*}
\frac{1}{\delta _{o}^{\varepsilon }\left( L\frac{1}{\delta _{o}}\right) ^{%
\frac{3}{2}}}\mathbb{P}\left( A_{\varepsilon ,\delta _{o}}^{c}\right)
\end{equation*}%
above by 
\begin{equation*}
\frac{9.57}{\varepsilon ^{3}}\mathbf{1}_{(0,2a]}(\varepsilon )+\left( \frac{%
14.59}{\varepsilon }+9.9\right) \mathbf{1}_{(2a,\infty ]}(\varepsilon )+24.05
\end{equation*}

Finally, for $W$ on the set $A_{\varepsilon ,\delta _{o}},$we let $0<\delta
\leq \delta _{0}$ and $s$ and $t$ be such that $0<\left\vert s-t\right\vert
<\delta .$ Let $m$ be the smallest integer so that $\delta \leq \delta _{m}.$
Clearly $m\geq n.$ By inequality \ref{tail and truncated part} and using the
same approach as in Theorem \ref{deviations for fixed delta (modulus)}, we
obtain 
\begin{align}
\left\vert W_{s}-W_{t}\right\vert & \leq \left\vert
W_{s}^{m}-W_{t}^{m}\right\vert +\left\vert \sum_{j=m+1}^{\infty
}2^{-j/2}\sum_{k=0}^{2^{j}-1}\left( \Lambda _{j,k}\left( t\right) -\Lambda
_{j,k}\left( s\right) \right) X_{j,k}\right\vert  \notag \\
& \leq \sqrt{1+\varepsilon }g(\delta )\left( 1+\frac{2.65}{\sqrt{L\frac{1}{%
\delta }}}\right) .  \notag
\end{align}%
The right hand side does not depend on our choice of $s$ and $t,$ so 
\begin{equation*}
\sup_{\left\vert s-t\right\vert \leq \delta }\frac{\left\vert
W_{s}-W_{t}\right\vert }{g\left( \delta \right) r\left( \delta \right) }\leq 
\sqrt{1+\varepsilon },
\end{equation*}%
which holds for every $\delta \leq \delta _{o}$ on the set $A_{\varepsilon
,\delta _{o}}$.
\end{proof}

\subsection{Consequences}

In this subsection we easily extend the results of the previous subsection
to Brownian motion on $\left[ 0,T\right] $ by using the scaling property of
Brownian motion.

\begin{corollary}
\label{MOD on [0,T] fixed delta}For $T\geq 1$ and $\delta \leq T2^{-5},$ we
have 
\begin{equation*}
\mathbb{P}\left( \underset{\left\vert s-t\right\vert \leq \delta }{%
\sup_{0\leq t<s\leq T}}\frac{\left\vert B_{s}-B_{t}\right\vert }{g(\delta
)r\left( \delta ,T\right) }\leq \sqrt{1+\varepsilon }\right) \geq
1-K_{1}\left( \varepsilon \right) \left( \frac{\delta }{T}\right)
^{\varepsilon }\left( L\frac{T}{\delta }\right) ^{\frac{3}{2}}
\end{equation*}%
and 
\begin{equation*}
\mathbb{P}\left( \sup_{\delta \leq \delta _{o}}\underset{\left\vert
s-t\right\vert \leq \delta }{\sup_{0\leq t<s\leq T}}\frac{\left\vert
B_{s}-B_{t}\right\vert }{g(\delta )r\left( \delta ,T\right) }\leq \sqrt{%
1+\varepsilon }\right) \geq 1-K_{2}\left( \varepsilon \right) \left( \frac{%
\delta _{o}}{T}\right) ^{\varepsilon }\left( L\frac{T}{\delta _{o}}\right) ^{%
\frac{3}{2}},
\end{equation*}%
where 
\begin{equation*}
r\left( \delta ,T\right) =r\left( \frac{\delta }{T}\right) \sqrt{\frac{L%
\frac{T}{\delta }}{L\frac{1}{\delta }}}
\end{equation*}%
and $K_{1},$ $K_{2}$ and $r$ are the same as in theorems \ref{deviations for
fixed delta (modulus)} and \ref{uniform mod theorem}.
\end{corollary}

\begin{proof}
For $T\geq 1$ and $\delta \leq T2^{-5},$ the scaling property of Brownian
motion yields 
\begin{align}
\mathbb{P}\left( \underset{\left\vert s-t\right\vert \leq \delta }{%
\sup_{0\leq t<s\leq T}}\frac{\left\vert B_{s}-B_{t}\right\vert }{g(\delta
)r\left( \delta ,T\right) }\leq \sqrt{1+\varepsilon }\right) & =\mathbb{P}%
\left( \underset{\left\vert s-t\right\vert \leq \delta /T}{\sup_{0\leq
t<s\leq 1}}\frac{\sqrt{T}\left\vert B_{s}-B_{t}\right\vert }{g(\delta
)r\left( \delta ,T\right) }\leq \sqrt{1+\varepsilon }\right)  \notag \\
& =\mathbb{P}\left( \underset{\left\vert s-t\right\vert \leq \delta /T}{%
\sup_{0\leq t<s\leq 1}}\frac{\left\vert B_{s}-B_{t}\right\vert }{g(\frac{%
\delta }{T})r\left( \frac{\delta }{T}\right) }\leq \sqrt{1+\varepsilon }%
\right) .  \label{[0,T] fixed delta}
\end{align}%
Similarly, 
\begin{eqnarray}
&&\mathbb{P}\left( \sup_{\delta \leq \delta _{o}}\underset{\left\vert
s-t\right\vert \leq \delta }{\sup_{0\leq t<s\leq T}}\frac{\left\vert
B_{s}-B_{t}\right\vert }{g(\delta )r\left( \delta ,T\right) }\leq \sqrt{%
1+\varepsilon }\right)  \notag \\
&=&\mathbb{P}\left( \sup_{\delta \leq \delta _{o}}\underset{\left\vert
s-t\right\vert \leq \delta /T}{\sup_{0\leq t<s\leq 1}}\frac{\left\vert
B_{s}-B_{t}\right\vert }{g(\frac{\delta }{T})r\left( \frac{\delta }{T}%
\right) }\leq \sqrt{1+\varepsilon }\right) .  \label{[0,t] uniform}
\end{eqnarray}%
The proof is complete by applying Theorem \ref{deviations for fixed delta
(modulus)} to equation \ref{[0,T] fixed delta} and Theorem \ref{uniform mod
theorem} to equation \ref{[0,t] uniform}.
\end{proof}

\section{Local Maximal Deviations for Truncated Brownian Motion and Brownian
Motion}

In this section, we develop new results regarding the local modulus of
continuity for Brownian motion. Our main contributions are finding estimates
for the distribution function of the maximum of the ratio of a truncated
Brownian motion process and the local modulus of continuity $h$ and the
maximum of the ratio of a Brownian motion process and the local modulus of
continuity $h.$ It may seem more simple to estimate the distribution
function of

\begin{equation*}
\sup_{t<\delta}\frac{W_{t}}{h\left( t\right) }
\end{equation*}
then of

\begin{equation*}
\underset{0\leq t<s\leq 1}{\sup_{\left\vert s-t\right\vert \leq \delta }}%
\frac{\left\vert W_{s}-W_{t}\right\vert }{g\left( \delta \right) }.
\end{equation*}%
However, the subtle difference of the local modulus being evaluated at $t$
and the global modulus being evaluated at $\delta $ makes establishing the
local case more challenging.

The first subsection contains one technical result designed to exploit the
fact that $W_{t}^{n}$ restricted to $I_{n+1,k}$ is a linear function in $t.$
It is an analog of Lemma \ref{mod calc lemma} adapted to treat the local
maximal deviations. In the second subsection, we detail the main results.

\subsection{Preliminaries}

As in the case of the modulus of continuity, we examine $W_{t}^{n}+\left(
W_{t}-W_{t}^{n}\right) .$ The following lemma is necessary in estimating the
maximal deviation for the truncated process. The restriction of $t_{2}<0.17$
appearing in the next lemma assures monotonicity of the function $f.$
Without this condition, the results obtained would still hold true but with
greater constants.

\begin{lemma}
\label{LIL Lemma}Let $a$ and $b$ be constants and $f:\left[ t_{1},t_{2}%
\right] \rightarrow\mathbb{R}$ , $t_{1}>0,$ $t_{2}<0.17$ be defined by 
\begin{equation*}
f(t)=\frac{at+b}{h(t)}
\end{equation*}
Then the relative maxima of$\ f$ must occur at the end points of $\left[
t_{1},t_{2}\right] $.
\end{lemma}

\begin{proof}
Calculus.
\end{proof}

\subsection{Local Maximal Deviations}

In this subsection we develop two results. First we estimate the probability
of the set 
\begin{equation*}
\left\{ \sup_{t\leq \delta }\frac{W_{t}^{n}}{h\left( t\right) }\geq \sqrt{%
1+\varepsilon }\right\} ,
\end{equation*}%
which, with a slight modification, gives an upper bound for the probability
of the set 
\begin{equation}
\left\{ \sup_{2^{-n-1}\leq t\leq 2^{-n}}\frac{W_{t}^{n}}{h\left( t\right) }%
\geq \sqrt{1+\varepsilon }\right\} .  \label{supOverIntervalJn}
\end{equation}%
Then, using the estimate of the size of the set \ref{supOverIntervalJn}
together with a modification of the uniform tail estimate established in
Lemma \ref{tail estimate}, we derive the main result, an upper bound for the
uniform maximal deviation from zero of the ratio of the Brownian motion
process and its modulus function $h\left( t\right) .$

\begin{remark}
The restriction $\delta \leq $ $2^{-4}$ imposed in this subsection serves
only one purpose: to make computations easier. Any other bound on $\delta $
which is less than one will only change the constants.
\end{remark}

\begin{theorem}
\label{LIL Truncated maximal deviations}Let $\varepsilon >0$, $n\in \mathbb{N%
}$ and $0<\delta \leq 2^{-4}.$ For $n$ such that $0<\delta <2^{-n-1},$ 
\begin{equation*}
\mathbb{P}\left( \sup_{t<\delta }\frac{W_{t}^{n}}{h\left( t\right) }\geq 
\sqrt{1+\varepsilon }\right) \leq \frac{\left( L\frac{1}{\delta }\right)
^{-1-\varepsilon }}{2\sqrt{\pi L_{2}\frac{1}{\delta }}},
\end{equation*}%
and for $n$ such that $\delta \geq 2^{-n-1},$ 
\begin{equation*}
\mathbb{P}\left( \sup_{t<\delta }\frac{W_{t}^{n}}{h\left( t\right) }\geq 
\sqrt{1+\varepsilon }\right) \leq \left( \left[ 2^{n+1}\delta \right]
+1\right) \frac{\left( L\frac{1}{\delta }\right) ^{-1-\varepsilon }}{\sqrt{%
\pi L_{2}\frac{1}{\delta }}}.
\end{equation*}
\end{theorem}

\begin{proof}
As in Theorem \ref{mod theorem nth partial sum}, we will use the fact that $%
W_{t}^{n}$\ is a linear function of $t$ when restricted to the intervals $%
I_{k}=I_{n+1,k},$ for $k=0,...,[2^{n+1}\delta].$ Thus 
\begin{equation}
\mathbb{P}\left( \sup_{t\leq\delta}\frac{W_{t}^{n}}{h\left( t\right) }\geq%
\sqrt{1+\varepsilon}\right) \leq\sum_{k=0}^{\left[ 2^{n+1}\delta\right] }%
\mathbb{P}\left( \sup_{t\in I_{n+1,k}}\frac{W_{t}^{n}}{h\left( t\right) }\geq%
\sqrt{1+\varepsilon}\right) .  \label{LIL nth partial sum at beginning}
\end{equation}

Set $\delta _{n}=\min \{\delta ,2^{-n-1}\}.$ We treat each $k$ differently:$%
\ k=0,$ $0<k<\left[ 2^{n+1}\delta \right] ,$ and $k=\left[ 2^{n+1}\delta %
\right] .$ Notice, when $k=0$, $W_{t}^{n}=tX_{o}+%
\sum_{j=0}^{n}2^{j/2}tX_{j,0}$ and thus 
\begin{equation}
\mathbb{P}\left( \sup_{t\in I_{0}}\frac{\sqrt{t}\left(
X_{o}+\sum_{j=0}^{n}2^{j/2}X_{j,0}\right) }{\sqrt{2\pi L_{2}\frac{1}{t}}}%
\geq \sqrt{1+\varepsilon }\right) \leq \frac{\left( L\frac{1}{\delta _{n}}%
\right) ^{-1-\varepsilon }}{2\sqrt{\pi L_{2}\frac{1}{\delta _{n}}}}\text{.}
\label{k=0 (LIL)}
\end{equation}%
And for $0<k<\left[ 2^{n+1}\delta \right] ,\ W_{t}^{n}$ can be written as $%
at+b$. By Lemma \ref{LIL Lemma}, we have the\ inequality%
\begin{align}
& \mathbb{P}\left( \sup_{t\in I_{k}}\frac{W_{t}^{n}}{h\left( t\right) }\geq 
\sqrt{1+\varepsilon }\right)  \notag \\
& \leq \mathbb{P}\left( \frac{W_{k2^{-n-1}}^{n}}{h\left( k2^{-n-1}\right) }%
\geq \sqrt{1+\varepsilon }\right) +\mathbb{P}\left( \frac{%
W_{(k+1)2^{-n-1}}^{n}}{h\left( (k+1)2^{-n-1}\right) }\geq \sqrt{%
1+\varepsilon }\right)  \notag \\
& \leq \frac{\left( L\frac{2^{n+1}}{k+1}\right) ^{-\left( 1+\varepsilon
\right) }}{\sqrt{\pi L_{2}\frac{2^{n+1}}{k+1}}}.  \label{k>0 (LIL)}
\end{align}%
Lastly, for $k=\left[ 2^{n+1}\delta \right] ,$ we apply Lemma \ref{LIL Lemma}
again to obtain 
\begin{equation}
\mathbb{P}\left( \sup_{t\in I_{\left[ 2^{n+1}\delta \right] }}\frac{W_{t}^{n}%
}{h\left( t\right) }\geq \sqrt{1+\varepsilon }\right) \leq \frac{\left( L%
\frac{1}{\delta }\right) ^{-1-\varepsilon }}{\sqrt{\pi L_{2}\frac{1}{\delta }%
}}.  \label{k=[2^(n+1)d]  (LIL)}
\end{equation}%
Next we look at small and large $\delta $ separately; that is, $0<\delta
<2^{-n-1}$ and $\delta \geq 2^{-n-1}.$ For $0<\delta <2^{-n-1},$we
incorporate inequalities \ref{k=0 (LIL)}, \ref{k>0 (LIL)} and \ref%
{k=[2^(n+1)d] (LIL)} into inequality \ref{LIL nth partial sum at beginning}
and see 
\begin{equation*}
\mathbb{P}\left( \sup_{t\leq \delta }\frac{W_{t}^{n}}{h\left( t\right) }\geq 
\sqrt{1+\varepsilon }\right) \leq \frac{\left( L\frac{1}{\delta }\right)
^{-1-\varepsilon }}{2\sqrt{\pi L_{2}\frac{1}{\delta }}}.
\end{equation*}%
Similarly for $\delta \geq 2^{-n-1},$ we obtain 
\begin{align*}
& \mathbb{P}\left( \sup_{t\leq \delta }\frac{W_{t}^{n}}{h\left( t\right) }%
\geq \sqrt{1+\varepsilon }\right) \leq \sum_{k=0}^{\left[ 2^{n+1}\delta %
\right] -1}\frac{\left( L\frac{2^{n+1}}{k+1}\right) ^{-1-\varepsilon }}{%
\sqrt{\pi L_{2}\frac{2^{n+1}}{k+1}}}+\frac{\left( L\frac{1}{\delta }\right)
^{-1-\varepsilon }}{\sqrt{\pi L_{2}\frac{1}{\delta }}} \\
& \leq \frac{\left( \left[ 2^{n+1}\delta \right] +1\right) \left( L\frac{1}{%
\delta }\right) ^{-1-\varepsilon }}{\sqrt{\pi L_{2}\frac{1}{\delta }}}.
\end{align*}
\end{proof}

We point out the nuances between the global and local cases briefly
mentioned in the introduction of this section. First in Theorem \ref{LIL
Truncated maximal deviations} the bounds obtained on the probability for the
truncated process are not summable over $n.$ Thus we can not directly employ
our methods used in Theorem \ref{uniform mod theorem}. Also recall that
according to Lemma \ref{tail estimate}, the tail behaves as $\sqrt{m2^{-m}}.$
At the end of the proof of Theorem \ref{uniform mod theorem} the tail term
in inequality \ref{tail and truncated part} is divided by $g\left(
m2^{-m}\right) $ resulting in a term that tends to zero as $m\rightarrow
\infty .$ This doesn't happen in the local case since we divide by $h\left(
t\right) ,$ not $h\left( \delta \right) $. To resolve these issues we
estimate the set%
\begin{equation*}
\left\{ \sup_{t\leq \delta }\frac{W_{t}^{n}}{h\left( t\right) }\geq \sqrt{%
1+\varepsilon }\right\}
\end{equation*}%
by breaking up the interval $\left( 0,\delta \right) $ into subintervals $%
J_{n}=\left[ 2^{-n-1},2^{-n}\right) $ where each subinterval $J_{n}$
produces an estimate which is a general term of a summable series. Also note
that breaking up the process over the intervals $J_{n}$ and choosing $d$
from Lemma \ref{tail estimate}$\ $to be $1+\frac{2}{\varepsilon }$ yields
terms of order of $n^{-1-\frac{\varepsilon }{2}}\left( Ln\right) ^{-\frac{1}{%
2}}$ for both the truncated process and the tail.

Denote 
\begin{equation*}
m\left( \varepsilon \right) =\left[ \frac{\varepsilon }{2L2}%
L_{2}2^{m+1}+f\left( \varepsilon \right) \right] +m+1
\end{equation*}%
where $f\left( \varepsilon \right) =\left( 1-\left( L2\right) ^{-1}\right) 
\mathbf{1}_{(0,1]}(\varepsilon ).$

\begin{corollary}
\label{LIL (nth sum corollary)}For $\varepsilon >0,$ $m\in \mathbb{N},$ 
\begin{equation*}
\mathbb{P}\left( \sup_{t\in J_{m}}\frac{W_{t}^{m\left( \varepsilon \right) }%
}{h\left( t\right) }\geq \sqrt{1+\varepsilon }\right) \leq \frac{2^{m\left(
\varepsilon \right) -m-1}\left( L2^{m+2}\right) ^{-\left( 1+\varepsilon
\right) }}{\sqrt{\pi \left( 1+\varepsilon \right) L_{2}2^{m+1}}}\text{.}
\end{equation*}
\end{corollary}

\begin{proof}
\begin{align*}
\mathbb{P}\left( \sup_{t\in J_{m}}\frac{W_{t}^{m\left( \varepsilon \right) }%
}{h\left( t\right) }\geq \sqrt{1+\varepsilon }\right) & \leq
\sum_{k=2^{m\left( \varepsilon \right) -m-1}}^{2^{m\left( \varepsilon
\right) -m}-1}\frac{\left( L\frac{2^{m\left( \varepsilon \right) +1}}{k+1}%
\right) ^{-\left( 1+\varepsilon \right) }}{\sqrt{\pi \left( 1+\varepsilon
\right) L_{2}\frac{2^{m\left( \varepsilon \right) +1}}{k+1}}} \\
& \leq \frac{2^{m\left( \varepsilon \right) -m-1}\left( L2^{m+2}\right)
^{-\left( 1+\varepsilon \right) }}{\sqrt{\pi \left( 1+\varepsilon \right)
L_{2}2^{m+1}}}
\end{align*}
\end{proof}

We are ready to determine the main results of the section, the local maximal
deviation of the ratio of the Brownian motion and the local modulus of
continuity.

\begin{theorem}
\label{LIL Maximal Deviation Theorem}For $0<\delta <2^{-4}$ and $\varepsilon
>0,$%
\begin{equation*}
\mathbb{P}\left( \sup_{t\leq \delta }\frac{W_{t}}{h\left( t\right) s\left(
t,\varepsilon \right) }\leq \sqrt{1+\varepsilon }\right) >1-J\left(
\varepsilon ,\delta \right)
\end{equation*}%
where 
\begin{equation*}
s\left( t,\varepsilon \right) =1+3.61\left( \frac{\mathbf{1}%
_{(0,1]}(\varepsilon )}{\sqrt{\varepsilon }\max \left\{ \sqrt{L_{2}\frac{1}{t%
}},\sqrt{\left( L\frac{1}{t}\right) ^{\frac{\varepsilon }{2}}}\right\} }+%
\frac{\mathbf{1}_{(1,\infty )}(\varepsilon )}{\left( L\frac{1}{t}\right) ^{%
\frac{\varepsilon }{4}}}\right)
\end{equation*}%
and 
\begin{equation*}
J\left( \varepsilon ,\delta \right) =\min \left( \frac{\frac{1.302}{%
\varepsilon }\mathbf{1}_{(0,1]}(\varepsilon )+1.18\mathbf{1}_{(1,\infty
)}(\varepsilon )}{\left( L\frac{1}{\delta }\right) ^{\frac{\varepsilon }{2}}%
\sqrt{L_{2}\frac{1}{\delta }}},1\right) .
\end{equation*}
\end{theorem}

\begin{proof}
As in Theorem \ref{deviations for fixed delta (modulus)}, the proof is
complete in two steps. First we estimate the size of the set 
\begin{equation*}
B_{\varepsilon ,\delta }=\left\{ \sup_{m\geq n}\sup_{t\in J_{m}}\frac{%
W_{t}^{m\left( \varepsilon \right) }}{h\left( t\right) }\leq \sqrt{%
1+\varepsilon }\right\} \cap \left\{ \sup_{m\geq n}\underset{0\leq k<2^{j-m}}%
{\max_{j>m\left( \varepsilon \right) }}\frac{\left\vert X_{j,k}\right\vert }{%
\sqrt{L2^{j-m-1}}}\leq 2\sqrt{1+\frac{1}{\varepsilon }}\right\} ,
\end{equation*}%
using both Corollary \ref{LIL (nth sum corollary)} and Lemma \ref{tail
estimate}. Then we show that on this set 
\begin{equation*}
\frac{W_{t}}{h\left( t\right) s\left( t,\varepsilon \right) }\leq \sqrt{%
1+\varepsilon }
\end{equation*}%
for all $t\leq \delta .$ Recall, we substitute $d=1+\frac{2}{\varepsilon }$
in Lemma \ref{tail estimate}.

Let $0<\delta <2^{-4}$ and choose $n$ so that $2^{-n-1}\leq \delta <2^{-n}$.
By Corollary \ref{LIL (nth sum corollary)} and Lemma \ref{tail estimate},we
establish that $\mathbb{P}\left( B_{\varepsilon ,\delta }^{c}\right) $ is no
greater than 
\begin{eqnarray}
&&\sum_{m=n}^{\infty }\left( \frac{2^{m\left( \varepsilon \right)
-m-1}\left( L2^{m+2}\right) ^{-\left( 1+\varepsilon \right) }}{\sqrt{\pi
\left( 1+\varepsilon \right) L_{2}2^{m+1}}}\right.  \notag \\
&&\left. +\frac{2^{-\left( 1+\frac{2}{\varepsilon }\right) \left( m\left(
\varepsilon \right) -m+1\right) }}{\left( 1-2^{-\left( 1+\frac{2}{%
\varepsilon }\right) }\right) \sqrt{\pi (2+\frac{2}{\varepsilon }%
)L2^{m\left( \varepsilon \right) -m+1}}}\right)  \notag \\
&\leq &\frac{1}{\sqrt{\pi \left( 1+\varepsilon \right) L_{2}2^{n+1}}}%
\sum_{m=n}^{\infty }\left( 2^{m\left( \varepsilon \right) -m-1}\left(
L2^{m+2}\right) ^{-\left( 1+\varepsilon \right) }\right.
\label{LILSecondProb} \\
&&\left. +2^{-\left( 1+\frac{2}{\varepsilon }\right) \left( m\left(
\varepsilon \right) -m+1\right) }\left( 1-2^{-\left( 1+\frac{2}{\varepsilon }%
\right) }\right) ^{-1}\right)  \notag
\end{eqnarray}%
To approximate this sum, we consider two cases: $0<\varepsilon \leq 1$ and $%
\varepsilon >1.$

Consider $0<\varepsilon \leq 1$. By definition of $m\left( \varepsilon
\right) $, we see that the summand for the infinite sum in expression \ref%
{LILSecondProb} is no greater than 
\begin{equation*}
2^{\frac{\varepsilon }{2L2}L_{2}2^{m+1}+1-\frac{1}{L2}}\left(
L2^{m+2}\right) ^{-\left( 1+\varepsilon \right) }+\frac{2^{-\left( 1+\frac{2%
}{\varepsilon }\right) \left( \frac{\varepsilon }{2L2}L_{2}2^{m+1}-\frac{1}{%
L2}+2\right) }}{\left( 1-2^{-\left( 1+\frac{2}{\varepsilon }\right) }\right) 
}.
\end{equation*}%
Thus the summand from expression \ref{LILSecondProb} is bounded above by 
\begin{equation*}
\left( m+2\right) ^{-\left( 1+\frac{\varepsilon }{2}\right) }2^{1-\frac{1}{L2%
}}+\frac{\left( m+1\right) ^{-\left( 1+\frac{\varepsilon }{2}\right)
}2^{-\left( 1+\frac{2}{\varepsilon }\right) \left( 2-\frac{1}{L2}\right) }}{%
\left( 1-2^{-\left( 1+\frac{2}{\varepsilon }\right) }\right) }.
\end{equation*}%
With some algebraic manipulation and by approximating a sum with an
appropriate integral we estimate expression \ref{LILSecondProb} from above
by 
\begin{equation*}
\frac{\left( L2\right) ^{-\left( 1+\frac{\varepsilon }{2}\right) }\left(
n+1\right) ^{-\frac{\varepsilon }{2}}}{\varepsilon \sqrt{\pi \left(
1+\varepsilon \right) L_{2}2^{n+1}}}\left( 2^{2-\frac{1}{L2}}+\frac{%
2^{-\left( 1+\frac{2}{\varepsilon }\right) \left( 1-\frac{1}{L2}\right) }}{%
2^{1+\frac{2}{\varepsilon }}-1}\left( 2+\frac{\varepsilon }{5}\right)
\right) .
\end{equation*}%
Since $0<\varepsilon \leq 1,$ 
\begin{equation*}
\frac{1}{2^{1+\frac{2}{\varepsilon }}-1}\leq \frac{8}{7}\frac{1}{2^{1+\frac{2%
}{\varepsilon }}}.
\end{equation*}%
and the function 
\begin{equation*}
\varepsilon \rightarrow \frac{2^{2-\frac{1}{L2}}+2^{-\left( 1+\frac{2}{%
\varepsilon }\right) \left( 2-\frac{1}{L2}\right) }\frac{8}{7}\left( 2+\frac{%
1}{5}\right) }{L\left( 2\right) \sqrt{\pi \left( 1+\varepsilon \right) }}
\end{equation*}%
attains an absolute maximum of no more than $1.302$ at $\varepsilon =1,$
producing the desired bound for expression \ref{LILSecondProb}. That is, 
\begin{equation*}
\mathbb{P}\left( B_{\varepsilon ,\delta }^{c}\right) \leq \frac{1.302}{%
\varepsilon \left( L\frac{1}{\delta }\right) ^{\frac{\varepsilon }{2}}\sqrt{%
L_{2}\frac{1}{\delta }}}
\end{equation*}%
when $0<\varepsilon \leq 1.$

Consider $\varepsilon >1.$ As in the previous case, we find an upper bound
for the summand of expression \ref{LILSecondProb}. Each summand is convex as
a function of $m\left( \varepsilon \right) .$ Recall that $m\left(
\varepsilon \right) $ changes definition when $\varepsilon >1.$ Therefore,
we may replace the greatest integer function in $m\left( \varepsilon \right) 
$ with either $\frac{\varepsilon }{2L2}L_{2}2^{m+1}$ or $\frac{\varepsilon }{%
2L2}L_{2}2^{m+1}-1.$ We determine which of these produces a greater value
for the summand of the infinite sum found in expression \ref{LILSecondProb}.
For $\frac{\varepsilon }{2L2}L_{2}2^{m+1}$ the summand is bounded above by 
\begin{equation}
\left( L2^{m+2}\right) ^{-\left( \frac{\varepsilon }{2}+1\right) }+\frac{%
\left( L2^{m+1}\right) ^{-\left( \frac{\varepsilon }{2}+1\right) }}{%
2^{\left( 1+\frac{2}{\varepsilon }\right) }\left( 2^{\left( 1+\frac{2}{%
\varepsilon }\right) }-1\right) }  \label{LIL gif = x}
\end{equation}%
and for $\frac{\varepsilon }{2L2}L_{2}2^{m+1}-1,$ the summand is bounded
above by 
\begin{equation}
\frac{\left( L2^{m+2}\right) ^{-\left( \frac{\varepsilon }{2}+1\right) }}{2}+%
\frac{\left( L2^{m+1}\right) ^{-\left( \frac{\varepsilon }{2}+1\right) }}{%
2^{\left( 1+\frac{2}{\varepsilon }\right) }-1}.  \label{LIL gif=x-1}
\end{equation}%
Comparing expressions \ref{LIL gif = x} and \ref{LIL gif=x-1}, we see the
former is greater. Thus the sum of expression \ref{LILSecondProb} is bounded
above by 
\begin{equation}
\sum_{m=n}^{\infty }\left( \left( L2^{m+2}\right) ^{-\left( \frac{%
\varepsilon }{2}+1\right) }+\frac{\left( L2^{m+1}\right) ^{-\left( \frac{%
\varepsilon }{2}+1\right) }}{2^{\left( 1+\frac{2}{\varepsilon }\right)
}\left( 2^{\left( 1+\frac{2}{\varepsilon }\right) }-1\right) }\right) .
\label{LIL e>1 prob bnd}
\end{equation}%
Again we use algebraic manipulation and approximate a sum with an
appropriate integral to see that we can bound expression \ref{LIL e>1 prob
bnd} above by 
\begin{equation*}
\frac{1}{\left( L2\right) ^{\left( \frac{\varepsilon }{2}+1\right) }\left(
n+1\right) ^{\frac{\varepsilon }{2}}}\left( \frac{5^{-1}}{2^{\left( 1+\frac{2%
}{\varepsilon }\right) }\left( 2^{\left( 1+\frac{2}{\varepsilon }\right)
}-1\right) }+\left( 1+\frac{1}{2^{\left( 1+\frac{2}{\varepsilon }\right)
}\left( 2^{\left( 1+\frac{2}{\varepsilon }\right) }-1\right) }\right) \frac{2%
}{\varepsilon }\right) .
\end{equation*}

Substitute $x=\frac{2}{\varepsilon }+1$ and note that the function%
\begin{equation*}
f\left( x\right) =\frac{5^{-1}}{2^{x}\left( 2^{x}-1\right) }+\left( 1+\frac{1%
}{2^{x}\left( 2^{x}-1\right) }\right) (x-1)
\end{equation*}%
is increasing for $x\in (1,3]$. Therefore $f(x)\leq f(3)\leq 2.04$ and hence%
\begin{equation*}
\mathbb{P}\left( B_{\varepsilon ,\delta }^{c}\right) \leq \frac{2.04}{\left(
L2\right) \left( L\frac{1}{\delta }\right) ^{\frac{\varepsilon }{2}}\sqrt{%
\pi \left( 1+\varepsilon \right) L_{2}\frac{1}{\delta }}}\leq \frac{1.18}{%
\left( L\frac{1}{\delta }\right) ^{\frac{\varepsilon }{2}}\sqrt{L_{2}\frac{1%
}{\delta }}}
\end{equation*}%
when $\varepsilon \in \lbrack 1,\infty ).$

We are now ready to determine an upper bound for 
\begin{equation*}
\frac{W_{t}}{h\left( t\right) \sqrt{1+\varepsilon }}
\end{equation*}%
on the set $B_{\varepsilon ,\delta }$ which will hold for all $t\leq \delta
. $ In the statement of this theorem, the desired upper bound is referred to
as $s\left( \delta ,\varepsilon \right) .$

Let $t\leq \delta $ and $m$ be such that $2^{-\left( m+1\right) }<t\leq
2^{-m}.$ By Corollary \ref{LIL (nth sum corollary)} and Lemma \ref{tail
estimate}, on the set $B_{\varepsilon ,\delta },$ we have 
\begin{align}
W_{t}& =W_{t}^{m\left( \varepsilon \right) }+\left( W_{t}-W_{t}^{m\left(
\varepsilon \right) }\right)  \notag \\
& \leq h\left( t\right) \sqrt{1+\varepsilon }+\frac{1}{2}\sum_{j=m\left(
\varepsilon \right) +1}^{\infty }2^{-j/2}\sqrt{L2^{j-m}}\underset{0\leq
k<2^{j-m}}{\sup_{j>m\left( \varepsilon \right) }}\frac{\left\vert
X_{j,k}\right\vert }{\sqrt{L2^{j-m}}}  \notag \\
& \leq h\left( t\right) \sqrt{1+\varepsilon }+\sqrt{1+\frac{1}{\varepsilon }}%
\sum_{j=m\left( \varepsilon \right) +1}^{\infty }2^{-j/2}\sqrt{L2^{j-m}}.
\label{LIL On Set}
\end{align}

By reindexing, the expression \ref{LIL On Set} is equivalent to 
\begin{equation*}
h\left( t\right) \sqrt{1+\varepsilon }\left( 1+g\left( \varepsilon ,t\right)
\right) ,
\end{equation*}%
where 
\begin{equation*}
g\left( \varepsilon ,t\right) =\sqrt{\frac{1}{\varepsilon }}\frac{1}{\sqrt{%
2tL_{2}\frac{1}{t}}}\sqrt{\frac{L2}{2^{m\left( \varepsilon \right) +1}}}%
\sum_{j=0}^{\infty }2^{-j/2}\sqrt{m\left( \varepsilon \right) +1+j-m}.
\end{equation*}

We focus our attentions on estimating the function $g\left( \varepsilon
,t\right) .$ As in our estimation of $B_{\varepsilon ,\delta }$, we consider
two cases,$0<\varepsilon \leq 1$ and $\varepsilon >1$.

\begin{remark}
In our estimations below we must remove the greatest integer function
appearing in $m\left( \varepsilon \right) .$There are simple methods for
doing so which result in worse constants. We provide a more detailed
computation which results in smaller constants.
\end{remark}

\begin{case}
$0<\varepsilon \leq 1.$

Consider $m\left( \varepsilon \right) =m.$ We break down computations
further by looking at the case where%
\begin{equation*}
L_{2}\frac{1}{t}\geq \left( L\frac{1}{t}\right) ^{\frac{\varepsilon }{2}}
\end{equation*}%
and the case where this inequality does not hold.

When $L_{2}\frac{1}{t}\geq \left( L\frac{1}{t}\right) ^{\frac{\varepsilon }{2%
}},$ $g\left( \varepsilon ,t\right) $ is bounded above by 
\begin{eqnarray*}
&&\frac{1}{\sqrt{2\varepsilon tL_{2}\frac{1}{t}}}\sqrt{\frac{L2}{2^{m+1}}}%
\sum_{j=0}^{\infty }2^{-j/2}\sqrt{1+j} \\
&\leq &\frac{1}{\sqrt{\varepsilon L_{2}\frac{1}{t}}}\sqrt{\frac{L2}{2}}%
\sum_{j=0}^{\infty }2^{-j/2}\sqrt{1+j}\leq \frac{3.46}{\sqrt{\varepsilon
L_{2}\frac{1}{t}}}.
\end{eqnarray*}

When $L_{2}\frac{1}{t}<\left( L\frac{1}{t}\right) ^{\frac{\varepsilon }{2}},$
we have $\left( L2^{m+1}\right) ^{\frac{\varepsilon }{2}}\leq \frac{e}{2}$
since $m\left( \varepsilon \right) =m$, . Also 
\begin{equation}
t\rightarrow \frac{\left( L\frac{1}{t}\right) ^{\frac{\varepsilon }{2}}}{%
tL_{2}\frac{1}{t}}  \label{DecreasingFuncoft}
\end{equation}%
is decreasing when $t<2^{-4}.$ (We use this fact repeatedly without mention
throughout the rest of the paper.)\ Thus $g\left( \varepsilon ,t\right) $ is
bounded above by%
\begin{equation*}
\frac{\sqrt{\left( L2^{m+1}\right) ^{\frac{\varepsilon }{2}}2^{m+1}}}{\sqrt{%
2\varepsilon \left( L\frac{1}{t}\right) ^{\frac{\varepsilon }{2}}L_{2}2^{m+1}%
}}\sqrt{\frac{L2}{2^{m+1}}}\sum_{j=0}^{\infty }2^{-j/2}\sqrt{1+j}\leq \frac{%
3.61}{\sqrt{\varepsilon \left( L\frac{1}{t}\right) ^{\frac{\varepsilon }{2}}}%
}.
\end{equation*}

Next consider $m\left( \varepsilon \right) =m+k,$ for $k\geq 1$ Again we
consider two cases, 
\begin{equation*}
L_{2}\frac{1}{t}\geq \left( L\frac{1}{t}\right) ^{\frac{\varepsilon }{2}}%
\text{ and }L_{2}\frac{1}{t}<\left( L\frac{1}{t}\right) ^{\frac{\varepsilon 
}{2}}.
\end{equation*}

When $L_{2}\frac{1}{t}\geq \left( L\frac{1}{t}\right) ^{\frac{\varepsilon }{2%
}}$, $g\left( \varepsilon ,t\right) $ is bounded above by 
\begin{equation}
\frac{1}{\sqrt{\varepsilon L_{2}\frac{1}{t}}}\sqrt{\frac{L2}{2^{k+1}}}%
\sum_{j=0}^{\infty }2^{-j/2}\sqrt{k+1+j}\leq \frac{2.87}{\sqrt{\varepsilon
L_{2}\frac{1}{t}}}  \label{LIL on set e>1, k>1}
\end{equation}%
since, for $k\geq 1,$ $j\geq 0,$ 
\begin{equation*}
k\rightarrow \frac{k+1+j}{2^{k}}
\end{equation*}%
is decreasing.

When $L_{2}\frac{1}{t}<\left( L\frac{1}{t}\right) ^{\frac{\varepsilon }{2}},$
we look at two subcases: $k=1$ and $k\geq 2.$

$k=1$ implies that the greatest integer function for $m\left( \varepsilon
\right) $ disappears. Here $g\left( \varepsilon ,t\right) $ is bounded above
by 
\begin{equation}
\frac{1}{\sqrt{\varepsilon \left( L\frac{1}{t}\right) ^{\frac{\varepsilon }{2%
}}}}\sqrt{\frac{\left( L2^{m+1}\right) ^{\frac{\varepsilon }{2}}}{%
2L_{2}2^{m+1}}}\sqrt{\frac{L2}{2}}\sum_{j=0}^{\infty }2^{-j/2}\sqrt{2+j}.
\label{LIL on set k=1}
\end{equation}%
With the restriction of $L_{2}\frac{1}{t}<\left( L\frac{1}{t}\right) ^{\frac{%
\varepsilon }{2}},$ the function 
\begin{equation*}
m\rightarrow \frac{\left( L2^{m+1}\right) ^{\frac{\varepsilon }{2}}}{%
L_{2}2^{m+1}}.
\end{equation*}%
on the set where $m\left( \varepsilon \right) =m+1$ is increasing. So we
substitute by the largest $m$ satisfying $m\left( \varepsilon \right) =1$
and get 
\begin{equation*}
\frac{\left( L2^{m+1}\right) ^{\frac{\varepsilon }{2}}}{\left(
L_{2}2^{m+1}\right) }\leq \frac{e}{\frac{2}{\varepsilon }}.
\end{equation*}%
Thus, expression \ref{LIL on set k=1}\ is bounded by $\frac{3.34}{\sqrt{%
\left( L\frac{1}{t}\right) ^{\frac{\varepsilon }{2}}}}$.

$k\geq 2$ implies that the greatest integer function appearing in $m\left(
\varepsilon \right) $ is greater than or equal to one. For $x=\frac{%
\varepsilon }{2}\log _{2}L2^{m+1}+1-\frac{1}{L2},$ the function 
\begin{equation*}
x\rightarrow \frac{x+2+j}{2^{x}}
\end{equation*}%
is decreasing for $x\geq 0,$ $j\geq 0.$ Thus, by substituting $\left[ x%
\right] $ with $x-1,$ $g\left( \varepsilon ,t\right) $ is bounded above by 
\begin{eqnarray*}
&&\sqrt{\frac{L2}{\varepsilon 2tL_{2}\frac{1}{t}\left( 2^{\frac{\varepsilon 
}{2L2}L_{2}2^{m+1}-\frac{1}{L2}+m+2}\right) }}\sum_{j=0}^{\infty }2^{-j/2}%
\sqrt{\frac{\varepsilon }{2L2}L_{2}2^{m+1}-\frac{1}{L2}+2+j} \\
&\leq &\frac{1}{\sqrt{\varepsilon \left( L\frac{1}{t}\right) ^{\frac{%
\varepsilon }{2}}}}\sqrt{\frac{L2}{2^{2-\frac{1}{L2}}}}\sum_{j=0}^{\infty
}2^{-j/2}\sqrt{\frac{\varepsilon }{2L2}+\frac{2-\frac{1}{L2}+j}{L_{2}2^{m+1}}%
} \\
&\leq &\frac{1}{\sqrt{\left( L\frac{1}{t}\right) ^{\frac{\varepsilon }{2}}}}%
\sqrt{\frac{L2}{2^{2-\frac{1}{L2}}}}\sum_{j=0}^{\infty }2^{-j/2}\sqrt{\frac{1%
}{2}\left( \frac{1}{L2}+\frac{2-\frac{1}{L2}+j}{1}\right) }\leq \frac{3.34}{%
\sqrt{\left( L\frac{1}{t}\right) ^{\frac{\varepsilon }{2}}}}.
\end{eqnarray*}
\end{case}

\begin{case}
$\varepsilon >1.$

For $x=\left[ \frac{\varepsilon }{2L2}L_{2}2^{m+1}\right] ,$ $x\rightarrow 
\frac{x+l}{2^{x}}$decreases as long as $l\geq 2$ and $x\geq 0.$ Therefore $%
g\left( \varepsilon ,t\right) $ is estimated by 
\begin{equation}
\sqrt{\frac{1}{\left( L\frac{1}{t}\right) ^{\frac{\varepsilon }{2}}}}\sqrt{%
\frac{L2}{2}}\sum_{j=0}^{\infty }2^{-j/2}\sqrt{\frac{1}{2L2}+\frac{1+j}{%
L_{2}2^{5}}}\leq \frac{3.59}{\sqrt{\left( L\frac{1}{t}\right) ^{\frac{%
\varepsilon }{2}}}}.  \label{LIL on set e>1 finalEstimate}
\end{equation}
\end{case}

Combining the various estimates from both cases, we have%
\begin{equation*}
W_{t}\leq h\left( t\right) \sqrt{1+\varepsilon }\left( 1+3.61\left( \frac{%
\mathbf{1}_{(0,1]}(\varepsilon )}{\sqrt{\varepsilon }\max \left\{ \sqrt{L_{2}%
\frac{1}{t}},\sqrt{\left( L\frac{1}{t}\right) ^{\frac{\varepsilon }{2}}}%
\right\} }+\frac{\mathbf{1}_{\left( 1,\infty \right) }(\varepsilon )}{\sqrt{%
\left( L\frac{1}{t}\right) ^{\frac{\varepsilon }{2}}}}\right) \right) .
\end{equation*}
\end{proof}

\begin{remark}
The authors note that in their paper \cite{DobricMarano}, the constants in
Proposition 2.1 can be improved upon slightly by using the results of
Theorem \ref{LIL Maximal Deviation Theorem}.
\end{remark}


\begin{thebibliography}{99}
\bibitem{DobricMarano} Dobric, V. and Marano, L.: Rates of convergence for L%
\'{e}vy's modulus of continuity and \u{H}in\u{c}in's law of the iterated
logarithm. High Dimensional Probability III (J. Hoffmann-J\o rgensen, M.
Marcus, and J. Wellner, eds.), Progress in Probability. \textbf{55}, Birkh%
\"{a}user, Basel, 105-109 (2003)

\bibitem{Einmahel} Einmahl, U.: The Darling-Erd\"{o}s theorem for sums of
i.i.d. random variables. Probab. Theory Related Fields \textbf{82(2)},
241-257 (1989)

\bibitem{Erdos} Erd\"{o}s, P.: On the law of the iterated logarithm. Ann.
Math. \textbf{43}, 419-436 (1942)

\bibitem{kolmogorov} Gnedenko, B. V. and Kolmogorov, A.N.: Limit
distributions for sums of independent random variables. Addison-Wesley
Publishing Company, Inc., Cambridge, Mass (1954)

\bibitem{Kahane} Kahane, J.P.: Some random series of functions. Cambridge
University Press, Cambridge (1985)

\bibitem{Khosh} Khoshnevisan, D., Levin, D. and Shi, Z.: Extreme-Value
Analysis of the LIL for Brownian Motion. Electron Commun Probab \textbf{10}
Paper \textbf{20}, 196-206 (2005)

\bibitem{levy} L\'{e}vy, P.: Th\'{e}orie de l'Additiondes Variables Al\'{e}%
atoires.\ Gauthier-Villars, Paris (1937)

\bibitem{meyer} Meyer, Y.: Ondelettes et Op\'{e}rateurs. Hermann, Paris
(1990)

\bibitem{pinsky} Pinsky, M.: A. Brownian continuity modulus via series
expansions. J. Theoret. Probab. Vol. \textbf{14}, No. \textbf{1}, 261--266
(2001)

\bibitem{Steele} Steele, M.: Stochastic Calculus and Financial Applications
(Stochastic Modelling and Applied Probability). Springer (2001)

\bibitem{talagrand} Talagrand, M. and Ledoux, M.: Probability in Banach
Spaces.\textit{\ }Springer-Verlag (1980)
\end{thebibliography}
\end{document}